\documentclass[11pt]{amsart}

\usepackage{hyperref,amsmath,amsthm,amsfonts,amscd,flafter,epsf,amssymb,
ifsym,color}
\usepackage[lmargin=1.5in,rmargin=1.5in,tmargin=1in,bmargin=1in]{geometry}
\usepackage{epsfig}
\usepackage{amsfonts}
\usepackage{amssymb,xfrac}
\usepackage{mymathabx}
\usepackage{amsmath}
\usepackage{amsthm}
\usepackage{epsf}

\usepackage{todonotes}
\input{diagram}

\newcommand{\smooth}{\mathrm{smooth}}
\newcommand{\topo}{\mathrm{top}}

\newcommand{\F}{\mathbb F}

\newcommand{\fmap}{\mathfrak{f}}

\newcommand{\tsf}{\mathsf{t}}

\newcommand{\Tangle}{\mathcal{T}}

\newcommand{\Rin}{\mathbb{M}}

\newcommand{\grt}{\mathrm{gr}_{\tsf}}

\newcommand{\nd}{\mathrm{nd}}

\newcommand{\la}{\mathfrak{u}}
\newcommand{\cl}{\mathfrak{c}}

\newcommand{\Ht}{\mathrm{H}}

\newcommand{\Qcal}{\mathcal{Q}}
\newcommand{\Dcal}{\mathcal{D}}

\newcommand{\el}{\kappa}
\newcommand{\Pcal}{\mathcal{P}}

\newcommand\Dual{\mathcal D}
\newcommand\Duality\Dual

\newcommand{\relspinc}{{\underline{\spinc}}}

\newcommand\x{\mathbf x}
\newcommand\w{\mathbf w}
\newcommand\z{\mathbf z}

\newcommand\y{\mathbf y}

\newcommand\ModSphere{\ModFlow\left({\mathbb S}\longrightarrow
\Sym^{g-1}(\Sigma_{1})\times \Sym^2(\Sigma_{2})\right)}
\newcommand\ModSpheres\ModSphere

\newcommand\gr{\mathrm{gr}}

\newcommand\UnparModSp{\widehat \ModSp}
\newcommand\UnparModFlow\UnparModSp

\newcommand{\spinc}{\mathfrak s}
\newcommand{\spinct}{\mathfrak t}

\newcommand\ModMaps{\mathcal M}
\newcommand\ModSp\ModMaps

\newcommand\Ta{{\mathbb T}_{\alpha}}
\newcommand\Tb{{\mathbb T}_{\beta}}
\newcommand\Tc{{\mathbb T}_{\gamma}}

\newcommand\del{\partial}

\newcommand\alphas{\mbox{\boldmath$\alpha$}}

\newcommand\betas{\mbox{\boldmath$\beta$}}
\newcommand\gammas{\mbox{\boldmath$\gamma$}}

\newcommand\Ring{\mathbb A}
\newcommand\Rring{\mathcal R}

\newcommand\spincrel\relspinc
\newcommand\relspinct{\underline{\mathfrak t}}

\newtheorem{convention}[equation]{Convention}
\newtheorem{thm}{Theorem}[section]
\newtheorem{prop}[thm]{Proposition}
\newtheorem{cor}[thm]{Corollary}

\newtheorem{lem}[thm]{Lemma}

\newtheorem{example}[thm]{Example}
\newtheorem{defn}[thm]{Definition}

\newtheorem{remark}[thm]{Remark}

\def\endproof{\relax\ifmmode\expandafter\endproofmath\else
  \unskip\nobreak\hfil\penalty50\hskip.75em\hbox{}\nobreak\hfil\bull
  {\parfillskip=0pt \finalhyphendemerits=0 \bigbreak}\fi}
\def\endproofmath$${\eqno\bull$$\bigbreak}
\def\bull{\vbox{\hrule\hbox{\vrule\kern3pt\vbox{\kern6pt}\kern3pt\vrule}\hrule}}

\newcommand{\R}{\mathbb{R}}

\newcommand{\Z}{\mathbb{Z}}

\newcommand{\ModSWfour}{\mathcal{M}}
\newcommand{\ModFlow}{\ModSWfour}

\newcommand{\SpinC}{{\mathrm{Spin}}^c}

\newcommand\abuts\Rightarrow
\newcommand\Sym{\mathrm{Sym}}

\newcommand{\Cob}{\mathcal{C}}

\newcommand{\CFT}{\mathrm{CF}}
\newcommand{\HFT}{\mathrm{HF}}

\newcommand{\m}{{\bf{m}}}
\newcommand{\match}{\mathfrak{m}}

\newcommand{\ra}{\rightarrow}
\newcommand{\Sig}{\Sigma}

\newcommand{\tvec}{\tsf} 

\newcommand{\HD}{\mathcal{H}} 

\begin{document}

\title[Upsilon invariant for graphs and applications]{Upsilon invariant for graphs and the homology cobordism group of homology cylinders}%
\author{Akram Alishahi}
\address{Department of Mathematics, University of Georgia, Athens, GA 30602}
\email{akram.alishahi@uga.edu}
\thanks{AA was supported by NSF Grant DMS-2019396.}

\begin{abstract}
Upsilon is a homomorphism on the smooth concordance group of knots defined by Ozsv\'{a}th, Stipsicz and Szab\'{o}. In this paper, we define a generalization of upsilon for a family of embedded graphs in rational homolog spheres. We show that our invariant will induce a homomorphism on the homology cobordism group of homology cylinders, and present some applications.  To define this invariant, we use tangle Floer homology. We lift relative gradings on tangle Floer homology to absolute gradings (for certain tangles) and prove a concatenation formula for it.
\end{abstract}
\maketitle
\tableofcontents


\section{Introduction}
Heegaard Floer theory was defined as an invariant of closed three-manifolds by Ozsv\'{a}th and Szab\'{o} \cite{OS-3m1,OS-3m2} and over the last two decades has developed into a powerful package of invariants for various low dimensional topology objects, including knots \cite{OS-knot, Ras1}, links \cite{OS-linkinvariant}, four-manifolds \cite{OS-4m}, contact structures \cite{OS-contact, HKM-1, HKM-2}, etc.  

Two important groups in low dimensional topology are \emph{homology cobordism group of homology three-spheres} $\Theta^3_{\Z}$ and \emph{smooth (resp. topological) knot concordance group} $\mathcal{C}^{\smooth}$ (resp. $\mathcal{C}^{\topo}$).  Using Heegaard Floer homology, Ozsv\'{a}th and Szab\'{o} define a homomorphism on $\Theta^3_{\Z}$ called \emph{$d$-invariant} or \emph{correction term} \cite{OS-d-invariant} which has applications in Dehn surgery problems and knot concordance. Several knot invariants have been developed for studying the concordance groups, such as $\tau$ \cite{OS-four-ball}, $\nu$ \cite{OS-nu}, $\nu^+=\nu^-$ \cite{Hom2014FourballGB} and $\Upsilon$ \cite{OSS-upsilon}. For instance, Ozsvath, Stipsicz and Szab\'{o} use $\Upsilon$ to construct a homomorphism from the smooth concordance group of knots $\mathcal{C}^{\smooth}$ to $\Z^\infty$ so that its surjective over the kernel of $\mathcal{C}^{\smooth}\to\mathcal{C}^{\topo}$ and reprove Hom's theorem \cite{Hom:summand-Ctop}.


The goal of this paper is to develop a generalization of upsilon for a family embedded graphs in rational homology spheres. 
We use our invariant to construct homomorphisms from $\mathcal{H}^{\smooth}_{0,n}$, the smooth homology cobordism group of homology cylinders over a genus zero surface with $n$ boundary components, to $\mathbb{R}$.  Our main tool is tangle Floer homology \cite{AE-1,AE-2}, and along the way we enrich tangle Floer homology by defining an absolute gradings on it for certain tangles, and proving a concatenation formula.

%


\begin{convention}
A bipartite graph is called \emph{balanced}, if both parts in each connected component of the graph have the same number of vertices. In this paper, graph means bipartite balanced graph where one part is labelled with $+$ and the other part with $-$, and an order is fixed on the edges of the graph. Specifically, if $G$ is connected with $2n$ vertices and $\el$ edges, vertices split as $v(G)=v_+(G)\coprod v_-(G)$ so that $v_{\pm}(G)$ has $n$ vertices, and the edges are labelled as $e_G=\cup_{i=1}^\el e_i$ connecting $v_-(G)$ to $v_+(G)$. For instance, if $G$ has two vertices and two edges and so a knot, such labelling will induce an orientation on the knot as $e_1\cup(-e_2)$. 
\end{convention}



Let $G$ be an embedded graph in a closed three-manifold $Y$ with $2n$ vertices and $\el$ edges. We define a system $L
_G$ of $2n$ linear equations with $\el$ variables by setting:
\begin{equation}\label{eq:LG}
L_G:=\left\{L_j:\sum_{v_j\in\del e_i}t_i=2~\Big|~v_j\in v(G)\right\}.
\end{equation}
A solution $\tsf=(t_1,t_2,\ldots,t_{\el})$ for $L_G$ is called \emph{non-negative} if every $t_i\ge 0$. Note that $L_G$ has a non-negative solution if and only if $G$ contains a perfect matching (Lemma \ref{lem:solconvex}). We denote the set of non-negative solutions for $L_G$ by $\mathsf{L}_{G}$. 

Suppose $Y$ is a rational homology sphere, $\mathsf{L}_G\neq\emptyset$ and $G$ is nullhomologous i.e. the image of the map induced by inclusion from $H_1(G,\Z)$ to $H_1(Y,\Z)$ is trivial. Then, for any $\spinc\in\SpinC(Y)$ we define a family of invariants 
\[\Upsilon_{G,\spinc}(\tsf):\Upsilon_1(\tsf)\le\Upsilon_2(\tsf)\le\cdots\le\Upsilon_{2^{n-1}}(\tsf).\]
where each $\Upsilon_i(\tsf)$ is a piecewise linear function over $\mathsf{L}_G$.  


\begin{thm} \label{thm:invhomcob}
Suppose $G$ and $G'$ are embedded null-homologous graphs in rational homology spheres $Y$ and $Y'$, respectively. For any $\spinc\in\SpinC(Y)$ and $\spinc'\in\SpinC(Y')$, if $(G,\spinc)$ and $(G',\spinc')$ are $\SpinC$ homology concordant (Definition \ref{def:homcob}) then $\Upsilon_{G,\spinc}(\tsf)=\Upsilon_{G',\spinc'}(\tsf)$ for any $\tsf\in\mathsf{L}_G=\mathsf{L}_{G'}$.
\end{thm}
An embedded graph $G$ is called a \emph{$\mathit{\Theta_n}$-type} graph if it has two vertices and $n$ edges. For any $\Theta_n$-type graph $G$ \[\mathsf{L}_G=\mathsf{L}_n=\{\tsf=(t_1,t_2,\cdots,t_n)\in[0,2]^n\ \big{|}\ t_1+t_2+\cdots+t_n=2\}.\]
We show that for $n=1,2$ our invariant recovers $d$-invariant and upsilon invariant for knots and links, respectively, as follows. \begin{enumerate}
\item If $n=1$, then then $\mathsf{L}_G=\{2\}$ and $\Upsilon_{G,\spinc}=d(Y,\spinc)$.
\item If $n=2$ and $Y=S^3$, then $\mathsf{L}_G=\{(t,2-t)\ |\ t\in [0,2]\}$ and $\Upsilon_{G}(t,2-t)=\Upsilon_K(t)$ where $K$ is the oriented knot associated with $G$.
\item If $n=2$ and $G$ is a disjoint union of $n$ $\Theta_2$-type graphs in $S^3$, then it represents an oriented link, denoted by $L$. The solution set $\mathsf{L}_G$ is an $n$-dimensional cube and under certain labeling of $e(G)$ as described in Example \ref{ex:link} $\Upsilon_G(t,t,\cdots,t, 2-t, 2-t,\cdots,2-t)=\Upsilon_L(t)$ for any $t\in[0,2]$.
\end{enumerate}


\begin{lem}\label{lem:sum} Suppose $G_1\subset Y_1$ and $G_2\subset Y_2$ are null-homologous $\Theta_n$-types graphs embedded in rational homology spheres. If $G\subset Y=Y_1\#Y_2$ is a vertex connected sum of $G_1$ and $G_2$ (Definition \ref{def:vsum}), we have
\[\Upsilon_{G,\spinc_1\#\spinc_2}(\tsf)=\Upsilon_{G_1,\spinc_1}(\tsf)+\Upsilon_{G_2,\spinc_2}(\tsf).\]
Here, $\spinc_1\in\SpinC(Y_1)$ and $\spinc_2\in\SpinC(Y_2)$.
\end{lem}

Let $S=S_{g,n}$ be a genus $g$ surface with $n$ boundary components. Homology cobordism group of homology cylinders over $S$ is an enlargement of mapping class group introduced by Garoufalidis and Levine \cite{GL, Le}. If $g=0$, every homology cylinder $X$ over $S_{0,n}$ is the complement of an embedded  $\Theta_n$-type graph $G$ in an integral homology sphere, and we define $\Upsilon_X(\tsf):=\Upsilon_{G}(\tsf)$.
\begin{prop}\label{prop:hom}
For any $\tsf\in\mathsf{L}_n$, mapping any homology cylinder $X$ over $S_{0,n}$ to $\Upsilon_{X}(\tsf)$ will induce a homomorphism from $\mathcal{H}_{0,n}^{\smooth}$ to $\R$.
\end{prop}

Building on \cite[Theorems 1.17 and 1.20]{OSS-upsilon}, we use the changes in the directional derivatives of $\Upsilon_X(\tsf)$ to construct a homomorphism from $\mathcal{H}_{0,n}^\smooth$ to $\Z^\infty$ whose restriction to kernel of $\mathcal{H}_{0,n}^{\smooth}\to \mathcal{H}_{0,n}^{\topo}$ is surjective (Lemmas \ref{lem:int} and \ref{lem:surj}). This will give a new proof for \cite[Theorem 1.1]{cha_friedl_kim_2011}.


On the other hand, a homology cylinder over $S_{0,n+1}$ corresponds to a string link with $n$ components in a homology cylinder over $D^2$, and these homology cylinders are homology cobordant if and only if the corresponding string links are homology concordant. Further, taking the closure of a homology string link we get an $n$ component link in an integral homology sphere. Follows from \cite{habegger_lin_1998} that two homology string links are homology cobordant if and only if their corresponding links are strongly homology concordant. Thus, $\Upsilon(\tsf)$ gives an obstruction for detecting whether two links are strongly homology concordant. 

For an embedded graph $G\subset Y$, the invariant $\Upsilon_{G,\spinc}$ is constructed by associating a tangle $(M_G,T_G)$ (in the sense of \cite{AE-2}) to $G$ and using tangle Floer homology defined by Eftekhary and the author \cite{AE-1}. Specifically, $M_{G}$ is constructed from $Y$ by removing pairwise disjoint balls around the vertices of $G$ and setting $T_G=G\cap M_{G}$. The splitting $v(G)=v_+(G)\coprod v_-(G)$ will induce a splitting $\del M_G=\del_+ M_G\coprod\del_-M_G$ and every connected component of $T_G$ is oriented from $\del_-M_G$ to $\del_+M_G$. If $Y$ is a rational homology sphere, for any $\tsf\in\mathsf{L}_G$, we define an absolute $\R$-grading on $\CFT(M_G,T_G,\spinc)$.  Then, following \cite{OSS-upsilon} we build a $\tsf$-modified tangle Floer chain complex denoted by $\tsf\CFT(M_G,T_G,\spinc)$ which is $\R$-graded. 

Additionally, to prove the vertex connected sum formula we show that tangle Floer homology has a concatenation formula as follows. 

\begin{thm}\label{TFH-concatenation}   Suppose $\Tangle_1$ and $\Tangle_2$ are $\Ring$-tangles corresponding to embedded graphs. For any diffeomorphism $d$ from $\del_+\Tangle_1$ to $\del_-\Tangle_2$ we have
\[\HFT(\Tangle)\cong H_*(\CFT(\Tangle_1)\otimes_{\Ring}\CFT(\Tangle_2))\]
where $\Tangle=\Tangle_1\circ_d\Tangle_2$.
\end{thm}

This paper is organized as follows. Section \ref{sec:background} is a brief overview on tangle Floer homology, and cobordisms between tangles. We continue our discussion of tangle Floer homology in Section \ref{sec:str} by first proving Theorem \ref{TFH-concatenation} in Section \ref{ssec:concat} and then defining an absolute $\R$-gradings on tangle Floer homology in torsion $\SpinC$ structure for tangles corresponding to embedded, null-homologous graphs with perfect matchings, Section \ref{sec:grading}. In Section \ref{sec:upsilon}, we define the $\tsf$-modified chain complex, the invariant upsilon, prove Theorem \ref{thm:invhomcob} and Lemma \ref{lem:sum}. Finally, in Section \ref{sec:app} we discuss the homology cobordism group of homology cylinders, the homomorphism induced on it by $\Upsilon$ (Proposition \ref{prop:hom}) and its applications.

\emph{Acknowledgments.} I am grateful to Robert Lipshitz for many inspiring conversations. I am also thankful to 
Eaman Eftekhary and Melissa Zhang for helpful conversations.

\section{Background: Tangle Floer homology}\label{sec:background}
In this Section we review some of the basic definitions of colored tangles, colored cobordisms between tangles and tangle Floer homology from \cite{AE-1} and \cite{AE-2}. 

A generalization of classical tangles to be in one-to-one correspondence with sutured manifolds without toroidal sutures is defined as follows.

\begin{defn}
A \emph{tangle} is a properly embedded oriented one-manifold $T$ in an oriented three-manifold with boundary $M$ so that $\del M=(-\del_-M)\coprod\del_+M$ and 
\begin{enumerate}
\item both $T$ and $M$ do not have a closed component,
\item each connected component $T_i\subset T$ is oriented from $\del_-T_i=T_i\cap \del_-M$ to $\del_+T_i=\del T_i\cap\del_+M$.
\end{enumerate}
Let $\del_{\pm}T=\coprod_i\del_\pm T_i$. We will think of $M$ as a cobordism from $\del_-M$ to $\del_+M$ and $T$
as a cobordism embedded in $M$ from $\del_-T$ to $\del_+T$. 
\end{defn}  
Suppose $M$ is connected. The tangle $(M,T)$ is called \emph{balanced} if $\chi(\del_-M)=\chi(\del_+M)$, and each connected component of $\del M$ has a non-empty intersection with $\del T$. If $M=\coprod_{i=1}^mM_i$ is not connected, then $(M,T)$ is called balanced if $(M_i,T_i=T\cap M_i)$ is balanced for all $i=1,\ldots, m$. Tangles are defined to be in one-to-one correspondence with sutured manifolds without toroidal sutures. Moreover, balanced tangles correspond to balanced sutured manifolds, see \cite{AE-2} for more details. 

Any embedded bipartite graph $G$ in a closed three-manifold $M$ specifies a tangle $(M_G,T_G)$ where $M_G$ is obtained from $M$ by removing pairwise disjoint ball neighborhoods around the vertices of $M$, and $T_G=G\cap M_G$. The tangle $(M_G,T_G)$ is balanced if and only if $G$ is balanced. 

\begin{convention} Over the rest of this paper, all tangles correspond to embedded graphs i.e. for all tangles $(M,T)$  every boundary component of $M$ has genus zero.
\end{convention}

\begin{defn}
Let $\F=\Z/2\Z$ and $\Ring$ be an $\F$-algebra. An \emph{$\Ring$-coloring} of a balanced tangle $(M,T=\coprod_{i=1}^{\el}T_i)$ with sphere boundary components is a labelling $\cl:\pi_0(T)\to \Ring$ s.t. for every connected component $M_i$ of $M$
\[\sum_{\del_+^{*}M\subset \del M_i} \cl(\del_+^{*}M)=\sum_{\del_-^{*}M\subset \del M_i} \cl(\del_-^{*}M).\]
Here, $\cl(\del_{\pm}^*M)=\prod_{T_j\cap \del_{\pm}^*M\neq\emptyset}\cl([T_j])$ where $\del_{\pm}^{*}M$ is a connected component of $\del_{\pm} M$.
\end{defn}

For each balanced tangle $(M,T)$ we define an $\F$-algebra $\Ring_T$ along with a \emph{universal} coloring $\cl_T:\pi_0(T)\to \Ring_T$ as follows. Specifically, every other coloring $\cl$ with some $\F$-algebra $\Ring$ is obtained as $\cl=f\circ \cl_T$ where $f$ is an algebra homomorphism from $\Ring_T$ to $\Ring$. Consider the polynomial ring $\F[\la_1,\la_2,\ldots,\la_{\el}]$. Each connected component $\del^*_{\pm}M$ of $\del M$ determines a monomial in this ring by setting
\[\la(\del_{\pm}^*M)=\prod_{T_j\cap \del_{\pm}^*M\neq\emptyset}\la_j.\]
Using these monomials we specify an ideal 
\[
\mathcal{I}:=\left\langle \sum_{\del_+^{*}M\subset M_i} \la(\del_+^{*}M)-\sum_{\del_-^{*}M\subset M_i} \la(\del_-^{*}M)\ |\ M_i\subset M\ \text{is a connected component} \right\rangle
\]
and we define

\begin{equation}\label{eg:tring}
\Ring_T=\frac{\F[\la_1,\la_2,\ldots,\la_{\el}]}{\mathcal{I}}.\end{equation}
The coloring $\cl_T$ is defined by setting $\cl_T([T_j])=\la_j$. It is easy to check that it is a coloring and it has the universal property. For the tangle $(M_G,T_G)$ corresponding to an embedded graph $G$, we denote $\Ring_G=\Ring_{T_G}$.

Let $\Theta_n$ be the connected bipartite graph with two vertices and $n$-edges. We say a tangle $(M,T)$ is a \emph{$\Theta_n$-type} tangle, if it corresponds to an embedding of $\Theta_n$ in a closed three-manifold. For any $\Theta_n$-type tangle \[\Ring_T=\F[\la_1,\cdots,\la_n].\]

An $\Ring$-tangle (with sphere boundary components) is a 4-tuple $\Tangle=[M,T,\cl,\spinc]$ where $(M,T)$ is a balanced tangle with sphere boundary components, $\cl$ is an $\Ring$-coloring on $(M,T)$ and $\spinc\in\SpinC(M)$.  A \emph{Heegaard diagram} (or \emph{$\Ring$-diagram})\[\HD=(\Sig,\alphas,\betas,\cl:\z\to\Ring,\spinc)\]
for $\Tangle$ consists of the followings.
\begin{enumerate}
\item An embedded surface $\Sig$ in $M$ which splits $M$ into two compression bodies $H_{\alpha}$ and $H_{\beta}$ where 
\[\del H_{\alpha}=-(\del_-M)\coprod \Sig\quad\text{and}\quad \del H_{\beta}=-\Sig\coprod\del_+M.\]
Moreover, $\Sig$ intersects each component $T_i$ of $T$ in exactly one point denoted by $z_i$, and $\z=\coprod_{i=1}^{\el}z_i$. 
\item $\cl:\z\to \Ring$ is the map induced by $\cl$, still denoted by $\cl$.
\item $\alphas=\coprod _{i=1}^{\ell}\alpha_i\subset \Sig$ is a union of embedded simple closed curves on $\Sig$ so that each $\alpha_i$ bounds a disk in $H_{\alpha}$ and after doing surgery on $\Sig$ along these disks we get a surface isotopic to $\del_-M$ in $M$.
\item $\betas=\coprod _{i=1}^{\ell}\beta_i\subset \Sig$ is a union of embedded simple closed curves on $\Sig$ so that each $\beta_i$ bounds a disk in $H_{\beta}$ and after doing surgery on $\Sig$ along these disks we get a surface isotopic to $\del_+M$ in $M$.
\end{enumerate}
Finally, we assume that $\HD$ is $\spinc$-admissible in the sense of \cite[Definition 4.1]{AE-1}.

Given a Heegaard diagram $\HD$ for an $\Ring$-tangle $\Tangle$ we define the chain complex $\CFT(\HD)$ to be the free $\Ring$-module generated by the intersection points $\x\in\Ta\cap\Tb$ in $\Sym^{\ell}(\Sig)$. Furthermore, the differential $\del:\CFT(\HD)\to\CFT(\HD)$ is defined by the formula
\[\del \x=\sum_{\y\in\Ta\cap\Tb}\sum_{\substack{\phi\in\pi_2(\x,\y)\\ \mu(\phi)=1}}\#\widehat{\mathcal{M}}(\phi)\cl(\phi)\cdot \y\]
where $\cl(\phi)=\prod_{i=1}^{\el}\cl(z_i)^{n_{z_i}(\phi)}$ and $\widehat{\mathcal{M}}(\phi)$ counts the number of holomorphic disks in $\Sym^{\ell}(\Sig)$ connecting $\x$ to $\y$ modulo $\R$-translation. Additionally, for any $\Ring$-module $\Rin$ we define the chain complex $\CFT^{\Rin}(\HD):=\CFT(\HD)\otimes_{\Ring}\Rin$. In \cite{AE-1} we prove that for any $\Ring$-module $\Rin$ the chain homotopy type of $\CFT^{\Rin}(\HD)$ is an invariant of the tangle $\Tangle$. So we denote 
\[\HFT(\Tangle):=H_{*}\left(\CFT(\HD)\right)\quad\text{and}\quad \HFT^{\Rin}(\Tangle):=H_{*}\left(\CFT^{\Rin}(\HD)\right).\]
Moreover, we denote $\CFT(M,T,\spinc):=\CFT(M,T,\cl_T,\spinc)$ and $\HFT(M,T,\spinc):=\HFT(M,T,\cl_T,\spinc)$ for the universal coloring $\cl_T$.

Note that for any fixed $\relspinc\in\spinc$, we have 
\[\CFT(M,T,\spinc)=\bigoplus_{\relspinc\in\spinc}\CFT(M,T,\relspinc).\]
Further, each summand $\CFT(M,T,\relspinc)$ is relatively graded by $\frac{\Z}{\mathfrak{d}(\relspinc)\Z}$ as an abelian group, where
\[\mathfrak{d}(\spinc)=\mathrm{gcd}_{h\in H_2(X)}\langle c_1(\relspinc),h\rangle.\] 
Here, $X=M\setminus\nd(T)$. To define this grading, consider a disk $\phi$ connecting the generator $\prod_{i=1}^{\el}\la_i^{a_i}\cdot\x$ to $\prod_{i=1}^{\el}\la_i^{b_i}\cdot\y$ with $n_{z_i}(\phi)=b_i-a_i$.Then, 
\begin{equation}\label{eq:gr}
\mathrm{gr}(\prod_{i=1}^{\el}\la_i^{a_i}\cdot\x,\prod_{i=1}^{\el}\la_i^{b_i}\cdot\y)=\mu(\phi).
\end{equation}
Since any two such disks differ by a periodic domain with coefficient zero at every $z_i$ i.e. a periodic domain which represents an element of $H_2(X)$, this grading is well-defined.

In \cite{AE-2}, Eftekhary and the author define TQFT type cobordism maps associated to $\Ring$-cobordisms between $\Ring$-tangles. We will recall some of the definitions we will need later for tangles with sphere boundary components. 

\begin{defn} 
Suppose $(M,T)$ and $(M',T')$ are balanced tangles corresponding to embedded graphs. A \emph{stable cobordism} from $(M,T)$ to $(M',T')$ is a pair $(W,F)$ satisfying the followings: 
\begin{enumerate}
\item $W$ is a smooth, oriented four-dimensional cobordism with boundary and corners from $M$ to $M'$, i.e. $\del W=\del_hW\cup\del_vW$ where
\[\del_vW=-M\coprod M',\quad \del_hW\cong [0,1] \times \del M\cong [0,1]\times \del M'\]
and $\del_vW\cap\del_hW=\del M\coprod\del M'$ are the corners.
\item Similarly, $F$ is an oriented, properly embedded surface with boundary and corners in $W$ which gives a cobordism from $T$ to $T'$ such that $\del_vF=(-T)\coprod T'$, $\del_h F\cap\del_hW=[0,1]\times\del T=[0,1]\times \del T'$ under the above identifications, and $\del_hF\cap\del_vF=\del T\cup\del T'$.
\item The orientation induced from $F$ on $\del F$ is the orientation inherited from $-T\coprod T'$.
\item (Stability)Every connected component of $F$ which is not a disk intersects both $T$ and $T'$ in more than one connected components.
\end{enumerate}
\end{defn}

\begin{defn}
Given a stable cobordism $(W,F)$ from $(M,T)$ to $(M',T')$ as above, an $\Ring$-coloring of $(W,F)$ is a labeling $\cl:\pi_0(F)\to \Ring$ so that it induces $\Ring$-colorings on $(M,T)$ and $(M',T')$ by inclusion. More precisely, if $i:\pi_0(T)\to\pi_0(F)$ and $i':\pi_0(T')\to\pi_0(F)$ are induced by inclusion, $\cl\circ i$ and $\cl\circ i'$ are $\Ring$-colorings for $(M,T)$ and $(M',T')$, respectively. Moreover, an \emph{$\Ring$-cobordism} is a four-tuple $\mathcal{C}=[W,F,\cl,\spinct]$ where $(W,F)$ is a stable cobordism equipped with an $\Ring$-coloring $\cl$ and $\SpinC$ class $\spinct\in\SpinC(W)$. We say, $\mathcal{C}$ is an $\Ring$-cobordism from $\Tangle=[M,T,\cl\circ i,\spinct|_{M}]$ to $\Tangle'=[M',T',\cl\circ i',\spinct|_{M'}]$.
\end{defn}

For any $\Ring$-cobordism $\mathcal{C}$ from $\Tangle$ to $\Tangle'$, an invariant homomorphism $\mathfrak{f}_{\mathcal{C}}$ from $\CFT(\Tangle)$ to $\CFT(\Tangle')$ is defined in \cite{AE-2}. As a special case, these maps recover cobordism maps defined by Ozsv\'{a}th and Szab\'{o} associated to cobordisms between closed three-manifolds in \cite{OS-4m}, see \cite[Section 8.1]{AE-2}. A similar statement holds for the cobordisms between tangles corresponding to three-manifolds with multiple basepoints. Specifically, assume $(M,T)$ and $(M',T')$ are tangles corresponding to multipointed closed, oriented three-manifolds i.e. each boundary component of $M$ (resp. $M'$) intersects exactly one component of $T$ (resp. $T'$). Then, as in \cite[Section 8.1]{AE-2} using $T$ one can define a natural correspondence
\[s_{T}:\SpinC(M)\to \SpinC(Y),\]
  where $Y$ is the closed, oriented three-manifold obtained by filling the boundary components of $M$ with three-handles. Further, if $T=\coprod_{i=1}^n T_i$ then $\Ring_T=\F[\la_1,\cdots,\la_n]$ and 
  \[\HFT(M,T,\spinc)\cong\HFT^-(Y,s_T(\spinc)).\]
  An analogous statement holds for $(M',T')$. 
 Next, assume every connected component of $F$ is a disk intersecting exactly one connected component of $T$ and $T'$. Then, filling $\del_h W$ by attaching copies of $[0,1]\times D^3$ we get a cobordism $Z$ from $Y$ to $Y'$, where $Y'$ is the closed, oriented three-manifolds obtained by filling the boundary components of $M'$ with three-handles. Moreover, using $F$ we get a natural correspondence $s_F$ between $\SpinC(W)$ and $\SpinC(Z)$, such that $s_T(\spinct|_{M})=s_F(\spinct)|_{Y}$ and $s_{T'}(\spinct|_{M'})=s_F(\spinct)|_{Y'}$ for any $\spinct\in\SpinC(W)$.  Let $\Ring=\Ring_T=\Ring_{T'}$ and equip $(W,F)$ with the $\Ring$-coloring $\cl$ that induces the universal colorings on $(M,T)$ and $(M',T')$, respectively. Then, for any $\spinct\in\SpinC(W)$, the associated map $\mathfrak{f}_{\mathcal{C}}$ to the cobordism $\mathcal{C}=[W,F,\cl,\spinct]$ is equal to a multipointed version of Ozsv\'{a}th-Szab\'{o} maps from $\HFT^-(Y,s_T(\spinct|_{M}))$ to $\HFT^-(Y',s_{T'}(\spinct|_{M'}))$.

\begin{defn} A stable cobordism $(W,F)$ from $(M,T)$ to $(M',T')$ is called a \emph{homology cobordism} if the following conditions hold:
\begin{itemize}
\item $W$ is a homology cobordism from $M$ to $M'$ i.e. homomorphisms $i_{*}:H_{j}(M)\to H_{j}(W)$ and $i'_{*}:H_{j}(M')\to H_j(W)$ induced by inclusion are isomorphisms for $0\le j\le 3$.
\item $F$ is a disjoint union of disks such that each component of $F$ intersects $T$ and $T'$ in exactly one component.
\end{itemize}
\end{defn}

 \section{More structure on tangle Floer homology} \label{sec:str}
As stated before, all tangles correspond to embedded graphs. In this section, first, we prove a concatenation formula for tangle Floer homology, and then we define relative and absolute gradings for it.
%

 \subsection{A concatenation formula}\label{ssec:concat}
 Fix an $\F$-algebra $\Ring$. Given $\Ring$-tangles  $\Tangle_1=[M_1,T_1,\spinc_1,\cl_1]$ and $\Tangle_2=[M_2,T_2,\spinc_2,\cl_2]$, by a \emph{diffeomorphism $d$ from $\del_+\Tangle_1$ to $\del_-\Tangle_2$} we mean a diffeomorphism
 \[d:(\del_+M_1,\del_+T_1)\to(\del_-M_2,\del_-T_2)\] 
 which respects the colorings. Specifically, $\cl_2\circ d_{*}=\cl_1$ where $d_{*}$ from $\pi_0(T_1)$ to $\pi_0(T_2)$ is the isomorphism induced by $d$. 
 
 Using any diffeomorphism $d$ from $\del_+\Tangle_1$ to $\del_-\Tangle_2$, we may concatenate $\Tangle_1$ and $\Tangle_2$ to construct a new tangle $\Tangle=\Tangle_1\circ_d\Tangle_2$. Note that $\Tangle=[M,T,\spinc,\la]$ where $M=M_1\cup_d M_2$, $T=T_1\cup_d T_2$, $\cl$ is induced by $\cl_1$ and $\cl_2$ and $\spinc|_{M_i}=\spinc_i$ for $i=1,2$.


For the rest of this section, we will focus on proving Theorem \ref{TFH-concatenation}. First, we define the special type of Heegaard diagrams that we will need for our proof. 

\begin{defn}\label{def:adaptedHD} A Heegaard diagram $\HD=(\Sig,\alphas,\betas,\cl:\z\to\Ring,\spinc)$ for $\Tangle$ is called \emph{adapted} to the concatenation $\Tangle=\Tangle_1\circ_d\Tangle_2$ if $\alphas=\alphas_1\coprod\alphas_2$ and $\betas=\betas_1\coprod \betas_2$ satisfying the followings.

\begin{enumerate}
\item $\alphas_2\cap\betas_1=\emptyset$. 
\item After surgery on $\Sig$ along $\alphas_2$ and $\betas_1$ we get an embedded surface, denoted by $\Sig[\alphas_2\cup\betas_1]$, isotopic to $\del_+M_1$ which is identified with $\del_-M_2$.
\item We obtain Heegaard diagrams $\HD_1$ and $\HD_2$ for $\Tangle_1$ and $\Tangle_2$, respectively, by setting:
\[\HD_1=(\Sig[\alphas_2],\alphas_1,\betas_1,\cl,\spinc_1)\quad\text{and}\quad\HD_2=(\Sig[\betas_1],\alphas_2,\betas_2,\cl,\spinc_2).\]
\end{enumerate}
\end{defn}

\begin{lem}\label{lem:admissible}
Any $\Ring$-tangle $\Tangle$ obtained by concatenating an $\Ring$-tangle $\Tangle_1$ with an $\Ring$-tangle $\Tangle_2$  accepts a Heegaard diagram adapted to the concatenation. \end{lem}

\begin{proof}
Let $f_i:M_i\to [-1,1]$ be a Morse function corresponding to a Heegaard diagram for $(M_i,T_i)$. Specifically, $f_i^{-1}(-1)=\del_-M_i$, $f_i^{-1}(1)=\del_+M_i$ and $f_i$ has no critical points with index zero or three. Moreover, $f_i$ evaluates $-1/2$  and  $1/2$ at critical points with index one and two, respectively.  Finally, $f_i|_{T_i}$ has no critical points. Let $f:M\to [-1,3]$ be the Morse function defined as $f|_{M_1}=f_1$ and $f|_{M_2}=f_2+2$, with possible smoothing near $\del_+M_1$ that is identified with $\del_-M_2$. 

Suppose $f_i$ has $\ell_i$ critical points with index one (and so index two).  Then, $f$ has $\ell_1+\ell_2$ critical points with index one (and so index two), but, $f$ is not ordered. We will fix this by switching the order of critical points of index 2 in $M_1$ with critical points of index one in $M_2$. This can be done by changing $f$ over a disjoint union of $\ell_1+\ell_2$ balls in $M$ 
to get an ordered Morse function $f'$ with the same critical points as $f$. Further, $f'$ evaluates $-\frac{1}{2}$ (resp. $\frac{5}{2}$) at critical points of index one (resp. two). Then $(f')^{-1}(1)$ is a Heegaard surface for $(M,T)$. Let $\z$ be the intersection of the tangle $T$ with the Heegaard surface and $\cl:\z\to\Ring$ be the coloring induced by the $\Ring$-coloring of $\Tangle$.  Fix a gradient like vector field for $f'$ and
let $\alphas$ (resp. $\betas$) be the union of the intersections of the unstable (resp. stable) disks for index one (resp. two) critical points with the Heegaard surface. We split $\alphas=\alphas_1\coprod\alphas_2$ (resp. $\betas=\betas_1\coprod\betas_2$) by setting $\alphas_i$ (resp. $\betas_i$) to be the curves corresponding to the critical points in $M_i$. It is easy to check that our Heegaard diagram satisfies in the conditions of Definition \ref{def:adaptedHD}. 

It is left to check that we can make  this diagram $\spinc$-admissible by applying isotopy to $\alpha$-curves, while keeping $\alphas_2$ disjoint from $\betas_1$. We will follow the proof of \cite[Lemma 4.5]{AE-1}, and adapt it to our setup. For $i=1,2$, we choose a set of pairwise disjoint simple closed curves $\gammas_i$ on the Heegaard diagram $\HD_i$ so that each element of this family is dual to at least one circle in $\alphas_i$ while every $\alphas_i$ has exactly one dual in $\gammas_i$. Moreover, for a sufficiently large $N$, \emph{winding} $\alphas_i$ around $\gammas_i$ $N$-times makes $\HD_i$ $\spinc$-admissible. Note that by winding $\alpha$ $N$-times along $\gamma$ we mean considering a parallel copy of $\gamma$ with reverse orientation, denoted by $\overline{\gamma}$, and winding $\alpha$ $N$-times along both $\gamma$ and $\overline{\gamma}$, see \cite[Lemma 4.5]{AE-1} for more details.

We may lift $\gammas_1\subset\Sig[\alphas_2]$ and $\gammas_2\subset\Sig[\betas_1]$ to simple closed curves on $\Sigma$, still denoted $\gammas_1$ and $\gammas_2$, so that $\gammas_1\cap \alphas_2=\emptyset$, $\gammas_2\cap\betas_1=\emptyset$ and $\gammas_1\cap\gammas_2=\emptyset$. We claim that for $N$ sufficiently large after winding the curves $\alphas=\alphas_1\coprod\alphas_2$ $N$-times along the curves $\gammas=\gammas_1\coprod\gammas_2$ we get an $\spinc$-admissible Heegaard diagram. If this claim does not hold, for each $n>0$ there is a periodic domain $\Pcal_n$ with $\langle c_1(\spinc),H(\Pcal_n)\rangle=0$ and $\cl(\Pcal_n)\neq 0$ so that after winding $\alphas$ along $\gammas$ $n$-times the resulted periodic domain $\Pcal_n'$ is positive. Coefficients of each domain represent a vector in $\Z^m$, where $m$ is the number of connected components of $\Sig\setminus(\alphas\cup\betas)$. For any periodic domain $\mathcal{P}$, we denote this vector by $\mathcal{P}$, as well. The sequence $\{\frac{\Pcal_n}{\|\Pcal_n\|}\}_{n=1}^{\infty}$ has a subsequence convergent to a unit vector in $\mathbb{R}^m$. Let $\Qcal$ denote the limit of this subsequence. If $\del_{\alpha}\Qcal$ has non-zero intersection number with some circle in $\gammas$, there exists $N>0$ such that after winding  the curves $\alphas$ along $\gammas$ $N$-times the resulting domain $\Qcal'$ will have some negative coefficients. So for $n$ sufficiently large $\Pcal_n'$ will not be positive. But, this contradicts the definition of $\Pcal_n$ and therefore, $\del_{\alpha}\Qcal$ will have intersection number zero with all of the $\gamma$-curves. Thus,
\[\Qcal=\sum_{i=1}^la_iA_i+\sum_{i=1}^kb_iB_i+\sum_{i=1}^{m}c_iC_i\]
where $\Sig\setminus \alphas=\bigcup_{i=1}^lA_i$, $\Sig\setminus \betas=\bigcup_{i=1}^kB_i$ and $\Sig\setminus \left(\alphas_2\cup\betas_1\right)=\bigcup_{i=1}^mC_i$. Here, $a_i,\ b_i$ and $c_i$ are nonnegative real numbers for every $i$. On the other hand, $\langle c_1(\spinc),H(\Qcal)\rangle=0$, and so 
\[\sum_{i=1}^l2a_i+\sum_{i=1}^k2b_i+\sum_{i=1}^m2c_i=0.\]
As a result, we have that $a_i, b_i$ and $c_i$ vanishes for all $i$, and so $\mathcal{Q}=0$. But this contradicts the fact that $\Qcal$ is a unit vector, and we are done.
\end{proof}

\begin{proof}(Theorem \ref{TFH-concatenation}) With Lemma \ref{lem:admissible} in place, the proof is similar to the proof of the connected sum formula for Heegaard Floer homology in \cite[Section 6]{OS-3m2}. 

Let $T=\coprod_{i=1}^{\el}T_i$, and label the components of $T_1=\amalg_{i=1}^{\el} T_1^i$ and $T_2=\amalg_{i=1}^{\el}T_2^i$ such that $T_*^i\subset T_i$. Fixing this labeling consider the universal algebras $\Ring_1=\Ring_{T_1}$ and $\Ring_2=\Ring_{T_2}$ corresponding to $(M_1,T_1)$ and $(M_2,T_2)$, respectively. If 
\[\Ring_1=\frac{\F[\la_1,\la_2,\cdots,\la_{\el}]}{\mathcal{I}_1}\quad\text{and}\quad\Ring_2=\frac{\F[\la_1,\la_2,\cdots,\la_{\el}]}{\mathcal{I}_2},\]
we define \[\tilde{\Ring}:=\frac{\F[\la_1,\la_2,\cdots,\la_{\el}]}{\mathcal{I}_1+\mathcal{I}_2}.\]
Clearly, $\tilde{\Ring}$ is a common quotient of $\Ring_T$, $\Ring_1$ and $\Ring_2$, and so there is a natural $\tilde{\Ring}$-colorings $\tilde{\cl}$ defined as 
\[\tilde{\cl}([T_i])=\la_i.\]
Moreover, restrictions of $\tilde{\cl}$ to $(M_1,T_1)$ and $(M_2,T_2)$ are $\tilde{\Ring}$-colorings. The given $\Ring$-coloring $\cl$ on $(M,T)$ factors through $\tilde{\cl}$ i.e.  for some homomorphism $f:\tilde{\Ring}\to\Ring$
we have $\cl=f\circ\tilde{\cl}$. Specifically, $f$ is defined by setting $f(\la_i)=\cl([T_i])$. Consequently, it is enough to prove the theorem for the coefficient ring $\tilde{\Ring}$ and  the coloring $\tilde{\cl}$, so we assume $\Ring=\tilde{\Ring}$ and $\cl=\tilde{\cl}$.

%


Let \[\HD=(\Sig,\alphas=\alphas_1\coprod\alphas_2,\betas=\betas_1\coprod\betas_2,\cl:\z\to\Ring,\spinc)\]
be a Heegaard diagram for $\Tangle$ adapted to the concatenation. Let $\gammas=\gammas_1\coprod\gammas_2$ be a set of pairwise disjoint simple closed curves on $\Sig$ such that $\gammas_1$ and $\gammas_2$ are obtained by applying small exact Hamiltonian isotopies to $\betas_1$ and $\alphas_2$, respectively. More precisely, each curve in $\gammas$ intersects the corresponding curve in $\betas_1\coprod\alphas_2$ in exactly two points while staying disjoint from the remaining curves. Let $(M_{\alpha\gamma},T_{\alpha\gamma})$ and $(M_{\gamma\beta},T_{\gamma\beta})$ denote the tangles associated with the Heegaard diagrams $(\Sig,\alphas,\gammas,\z)$ and $(\Sig,\gammas,\betas,\z)$. The Heegaard triple $\left(\Sig,\alphas, \gammas,\betas,\z\right)$ specifies a cobordism $W_{\alpha\gamma\beta}$ from $(M_{\alpha\gamma},T_{\alpha\gamma})\coprod (M_{\gamma\beta},T_{\gamma\beta})$ to $(M,T)$. Moreover, $\cl$ induces an $\Ring$-coloring on this cobordism.  

Note that there is a cobordism $W_1$ from $(M_1,T_1)$ to $(M_{\alpha\gamma},T_{\alpha\gamma})$ obtained by attaching $|\alphas_2|$ one-handles. Similarly, there is a cobordism $W_2$ from $(M_2,T_2)$ to $(M_{\gamma\beta},T_{\gamma\beta})$ by attaching $|\beta_1|$ one-handles. Moreover, composing $W_{\alpha\beta\gamma}$ with $W_1$ and $W_2$ we get the identity cobordism from $(M_1\cup_dM_2,T_1\cup_d T_2)=(M,T)$ to itself. Therefore,  $\spinc$ induces a natural $\SpinC$ structure on this product cobordism whose restriction to $W_{\alpha\gamma\beta}$, $W_1$ and $W_2$ is denoted by $\mathfrak{t}$, $\mathfrak{t}_1$ and $\mathfrak{t}_2$, respectively.

Counting holomorphic triangles representing $\mathfrak{t}$ we get a chain map 
\[f_{\alpha\gamma\beta,\mathfrak{t}}:\CFT(M_{\alpha\gamma},T_{\alpha\gamma},\cl,\spinc_1^0)\otimes\CFT(M_{\gamma\beta},T_{\gamma\beta},\cl,\spinc_2^0)\to\CFT(M,T,\cl,\spinc).\]
Here, $\spinc_1^0=\mathfrak{t}_1|_{M_{\alpha\gamma}}$ and $\spinc_2^0=\mathfrak{t}_2|_{M_{\beta\gamma}}$.

Furthermore, let $g_1=g_{W_1,\mathfrak{t}_1}$ and $g_2=g_{W_2,\mathfrak{t}_2}$ denote chain maps corresponding to cobordisms for $W_1$ and $W_2$, respectively. More precisely, 
\[g_1:\CFT(\Sig[\alphas_2],\alphas_1,\gammas_1,\cl,\spinc_1)\to\CFT(\Sig,\alphas,\gammas,\cl,\spinc_1^0)\] is given by
$g_1(\x_1')=\x_1'\otimes\Theta_{\alpha_2\gamma_2}$. Here, $\Theta_{\alpha_2\gamma_2}$ denotes the top intersection point between $\alphas_2$ and $\gammas_2$ curves and $g_1$ is a chain map for an appropriate choice of complex structure on $\Sig$. Similarly, 
\[g_2:\CFT(\Sig[\betas_1],\gammas_2,\betas_2,\cl,\spinc_2)\to\CFT(\Sig,\gammas,\betas,\cl,\spinc_2^0)\]
is given by $g_2(\x_2')=\Theta_{\gamma_1\beta_1}\otimes\x_2'$. Since $\alphas_2\cap\betas_1=\emptyset$ the complex structure on $\Sig$ can be chosen such that both $g_1$ and $g_2$ are chain maps. Let $g=g_1\otimes g_2$ and $\Gamma=f_{\alpha\gamma\beta,\mathfrak{t}}\circ (g_1\otimes g_2)$.


Next, we use an area filtration argument to show that $\Gamma$ is a quasi-isomorphism.  Since the Heegaard diagram $\HD$ is admissible, we may choose a volume form for $\Sig$ such that all of the periodic domains $\Pcal$ with $\langle c_1(\spinc),H(\Pcal)\rangle=0$ have total signed area zero. Further, consider the corresponding area forms on $\Sig[\alphas_2]$ and $\Sig[\betas_1]$.

It is straightforward that for $\Ring=\tilde{\Ring}$ and $\cl=\tilde{\cl}$ tangle Floer homology complex splits over the relative $\SpinC$ structures i.e.
\[\CFT(M,T,\cl,\spinc)=\bigoplus_{\relspinc\in\spinc}\CFT(M,T,\cl,\relspinc)\quad\text{and}\quad\CFT(M_i,T_i,\cl_i,\spinc_i)=\bigoplus_{\relspinc\in\spinc_i}\CFT(M_i,T_i,\cl_i,\relspinc)\]
for $i=1,2$. Further, each summand is a relatively graded by $\frac{\Z}{\mathfrak{d}(\relspinc)\Z}$ and $\frac{\Z}{\mathfrak{d}(\relspinc_i)\Z}$ as in Equation \ref{eq:gr}, respectively. It is easy to see that if $\relspinc=\relspinc_1\cup\relspinc_2$ then $\mathfrak{d}(\relspinc)$ divides $\mathfrak{d}(\relspinc_1)$ and $\mathfrak{d}(\relspinc_2)$. Thus, one may equip $\CFT(M_1,T_1,\cl_1,\relspinc_1)$ and $\CFT(M_2,T_2,\cl_2,\relspinc_2)$ with relative $\frac{\Z}{\mathfrak{d}(\relspinc)\Z}$-grading and $\Gamma$ induces a chain map 
\[\CFT(M_1,T_1,\cl_1,\relspinc_1)\otimes\CFT(M_2,T_2,\cl_2,\relspinc_2)\to\CFT(M,T,\cl,\relspinc) \]
that preserves these relative gradings.


Given a $\relspinc\in\spinc$, we define a filtration on each relative degree of $\CFT(\Sig,\alphas,\betas,\cl,\relspinc)$ by setting
\[\mathcal{F}(\prod_{i=1}^{\el}\la_i^{a_i}\cdot\x,\prod_{i=1}^{\el}\la_i^{b_i}\cdot\y)=-\mathcal{A}(\mathcal{D}(\phi)),\]
where $\phi\in\pi_2(\x,\y)$ with $\mu(\phi)=0$ and $n_{z_i}(\phi)=b_i-a_i$. This is well-defined, because any two such disks differ in a periodic domain $\Pcal$ with $\mu(\Pcal)=\langle c_1(\spinc),H(\Pcal)\rangle=0$ which has $\mathcal{A}(\Pcal)=0$. 
Furthermore, this filtration is bounded in each summand. Because, if generators $\prod_{i=1}^{\el}\la_i^{a_i}\cdot\x$ and $\prod_{i=1}^{\el}\la_i^{b_i}\cdot\x$ have the same relative grading then there is a disk $\phi\in\pi_2(\x,\x)$ with Maslov index zero such that $n_{z_i}(\phi)=b_i-a_i$. By assumption $\mathcal{A}(\mathcal{D}(\phi))=0$ and thus $\mathcal{F}(\prod_{i=1}^{\el}\la_i^{a_i}\cdot\x)=\mathcal{F}(\prod_{i=1}^{\el}\la_i^{b_i}\cdot\x)$. Analogously, we define a filtration on each relative degree summand of $\CFT(\Sigma[\alphas_2],\alphas_1,\gammas_1,\cl_1,\relspinc_1)$ and $\CFT(\Sigma[\betas_1],\gammas_2,\betas_2,\cl_2,\relspinc_2)$.

Assume that the total unsigned area between each $\betas_1$ or $\alphas_2$ curve and its Hamiltonian isotope is smaller than a sufficiently small $\epsilon$. Any $\x_1'\in \mathbb{T}_{\alpha_1}\cap\mathbb{T}_{\gamma_1}$ corresponds to a unique closest intersection point $\x_1\in \mathbb{T}_{\alpha_1}\cap\mathbb{T}_{\beta_1}$. Similarly, any $\x_2'\in \mathbb{T}_{\gamma_2}\cap\mathbb{T}_{\beta_2}$ corresponds to a unique closest intersection point $\x_2\in \mathbb{T}_{\alpha_2}\cap\mathbb{T}_{\beta_2}$. Define $\Gamma_0(\x_1'\otimes \x_2')=\x_1\cup\x_2$. Then, $\Gamma_0$ is clearly an isomorphism in each degree, but not necessarily a chain map. Moreover,    
\[\Gamma=\Gamma_0+\text{lower order}\]
with respect to the filtration $\mathcal{F}$. Because, first of all there is a canonical small triangle $\psi_0\in \pi_2(\x_1'\otimes \Theta_{\alpha_2\gamma_2}, \Theta_{\gamma_1\beta_1}\otimes \x_2',\x_1\otimes\x_2)$ with $\mu(\psi_0)=0$ contributing with coefficient one to $f_{\alpha\gamma\beta,\spinct}$. Then, if a triangle $\psi\in\pi_2(\x_1'\otimes \Theta_{\alpha_2\gamma_2}, \Theta_{\gamma_1\beta_1}\otimes \x_2',\y_1\otimes\y_2)$ with $\mu(\psi)=0$ has a holomorphic representative which contributes to $f_{\alpha\gamma\beta,\spinct}$ then there exists a disk $\phi$ connecting $\x_1\otimes\x_2$ to $\y_1\otimes \y_2$ with $\mu(\phi)=0$. Moreover, $\phi$ can be chosen such that the difference of $\psi$ and $\psi_0\ast\phi$ is a doubly periodic domain, whose domain is supported in the isotopy regions between circles in $\alphas_2$ and $\betas_1$ and their corresponding isotopes in $\gammas_2$ and $\gammas_1$, respectively. Therefore, $\mathcal{A}(\psi)=\mathcal{A}(\psi_0\ast \phi)$. Note that for a sufficiently small $\epsilon$, area of $\psi$ is bigger than $\epsilon$, while the area of $\psi_0$ is smaller than $\epsilon$. So $\mathcal{A}(\phi)>0$. Sinc the filteration is bounded below, and $\Gamma$ induces an isomorphism in each degree, it is an isomorphism.

\end{proof}

\subsection{Absolute Grading}\label{sec:grading}
Let $(M,T)=(M_G,T_G)$ be the tangle corresponding to an embedded graph $G$. In this section, we discuss relative and absolute $\R$-gradings on $\CFT(M,T,\spinc)$ for certain \emph{torsion} classes $\spinc\in\SpinC(M)$.

Suppose $G$ has $2n$ vertices and $\el$ edges. Every solution $\tsf=(t_1,\cdots,t_{\el})$ for the linear system $L_G$ defined in Equation \ref{eq:LG} gives an $\R$-grading, denoted by $\grt$, on $\Ring_T$ by setting $\grt(\la_i)=t_i$. So $t_i$ can be thought as a weight on $T_i$, and we use the notation $\mathsf{L}_T$ interchangeably with $\mathsf{L}_G$.

%

 Recall that a \emph{perfect matching} for $G$ is a subset of edges such that each vertex is incident to exactly one of them. Any perfect matching $\match=\coprod_{i=1}^ne_{j_i}$ specifies an element $\mathsf{t}_{\match}\in\mathsf{L}_G$ by setting 
 \[t_j=
\begin{cases}\begin{split} 
&2&\quad&j\in\{j_1,\ldots,j_{n}\}\\
&0&\quad&\text{otherwise}. 
\end{split}
\end{cases} 
\]

\begin{lem}\label{lem:solconvex} The subset $\mathsf{L}_G\subset \R^{\el}$ is the convex hull of the points \[\{\tsf_{\match}\ |\ \match:\ \text{perfect matching for}\ G\}.\]
In particular, $L_G$ has a non-negative solution if and only if $G$ contains a perfect matching. Furthermore, if $G$ is connected, this solution is unique if and only if $G$ is a tree.
\end{lem}
\begin{proof}
First of all, $\mathsf{L}_G$ is clearly convex. Let $\tsf\in\mathsf{L}_G$. If every coordinate of $\tsf$ is either equal to zero or two, then $\match=\coprod_{t_i=2}e_i$ is a perfect matching for $G$ and $\tsf=\tsf_{\match}$. Otherwise, there exists a loop $\gamma=e_{j_1}e_{j_2}\ldots e_{j_{2k}}$ such that $t_{j_i}\neq 0,2$ for every $i=1,\ldots, 2k$. Let 
\[t_o=\min\{t_{j_{2i+1}}\ |\ 0\le i\le k-1\}\quad \text{and} \quad t_{e}=\min\{t_{j_{2i}}\ |\ 1\le i\le k\}.\]
We construct new non-negative solutions $\tsf'=(t_1',\ldots,t_{\el'})$ and $\tsf''=(t_1'',\ldots,t_{\el}'')$ of $L_{G}$ by modifying $\tsf$ at $t_{j_1},\ldots,t_{j_{2k}}$ as follows. Let 

\begin{center}
\begin{minipage}[t]{0.4\textwidth}
\[t'_{j_i}=\begin{cases}\begin{array}{lcl}
t_{j_i}-t_{e}&\quad&i\ \text{even}\\
t_{j_i}+t_{e}&\quad&i\ \text{odd},
\end{array}
\end{cases}\]\end{minipage}
\begin{minipage}[t]{0.4\textwidth}
\[t''_{j_i}=\begin{cases}\begin{array}{lcl}
t_{j_i}+t_{o}&&i\ \text{even}\\
t_{j_i}-t_{o}&&i\ \text{odd}.
\end{array}
\end{cases}\]
\end{minipage}  
\end{center}

and $t_j'=t_{j}''=t_j$ for $j\notin\{j_1,\ldots, j_{2k}\}$. Note that both $\tsf'$ and $\tsf''$ have at least one more edge with weight zero or two comparing to $\tsf$. Moreover, $\tsf$ is on the line segment connecting $\tsf'$ to $\tsf''$ i.e. it is in their convex hull. Repeat this for $\tsf'$ and $\tsf''$. We will continue this process until we find a set of solutions where all of them correspond to perfect matchings for $G$. Since $\tsf$ belongs to their convex hull, we are done. 
\end{proof}

\begin{defn}  Suppose $\mathsf{L}_G\neq\emptyset$. For a $\tsf=(t_1,t_2,...,t_{\el})\in \mathsf{L}_G$, a $\SpinC$ structure $\spinc\in\SpinC(M)$ is called \emph{$\tsf$-torsion} if for any $H\in\Ht_2(M)$
$$\langle c_1(\spinc), H\rangle= \sum_{i=1}^{\el}t_i([T_i] \cdot H).$$
where $[T_i]\in H_1(M,\del M)$ is the homology class represented by $T_i$.
\end{defn}

For example, suppose $G$ is embedded in $S^3$ and $\del^\bullet M_{G}=\coprod_{i=1}^n R_i^\bullet$ for $\bullet\in\{+,-\}$. Then, for any $\tsf\in \mathsf{L}_G$ there is a unique $\tsf$-torsion $\SpinC$ structure, which is the class $\spinc$ where $\langle c_1(\spinc), [R_i^+]\rangle=2$ and $\langle c_1(\spinc), [R_i^-]\rangle=-2$ for every $i$. 

\begin{lem}\label{lem:tor-ind}
If the embedded graph $G$ is null-homologous, the set of $\tsf$-torsion $\SpinC$ structures on $M_G$ does not depend on the choice of $\tsf$. 
\end{lem}

\begin{proof}
Fix a perfect matching $\match$ for $G$. We show that the set of $\tsf$-torsion $\SpinC$ classes is equal to the set of $\tsf_{\match}$-torsion classes. If $G$ is a tree, $\tsf=\tsf_{\match}$ and we are done. Let $S$ be a spanning tree for $G$ such that $\match\subset S$. Choose an edge $e_i$ not included in $S$ such that $t_i\neq 0$. Connect the adjacent vertices of $e_i$ with a path $e_{i_1}\ldots e_{i_k}$ in $T$ to get a loop $e_ie_{i_1}\ldots e_{i_k}$ in $G$. Then, we modify $\tsf$ to get a new solution $\tsf'=(t_1',t_2',\ldots,t_{\el}')$ for $L_G$ by setting 
\[
t_j'=\begin{cases}\begin{array}{lll}
t_j&&j\neq i,i_1,i_2,\ldots, i_k\\
0&&j=i\\
t_{i_j}-t_i&& \text{if}\ j\ \text{even}\\
t_{i_{j}}+t_i&&\text{if}\ j\ \text{odd}
\end{array}
\end{cases}.
\]
Note that $\tsf'$ is not necessarily non-negative. Since $G$ is null-homologous, $[T_{i_j}]\cdot H=[T_{i_{j+1}}]\cdot H$ for any $H\in H_2(M_G)$, where $1\le j\le k$ and $e=e_{i_{k+1}}$. As a result,  
\[\sum_{i=1}^{\el}t_i([T_i]\cdot H)=\sum_{i=1}^\el t_i'([T_i]\cdot H).\]
We will continue this process until we get to a solution $\tilde{\tsf}$ for $L_G$ such that all of its non-zero coordinates correspond to the edges included in $S$. Therefore, $\tilde{\tsf}=\tsf_{\match}$ and so an $\spinc\in\SpinC(M_{G})$ is $\tsf$-torsion if and only if it is $\tsf_{\match}$-torsion.
\end{proof}

Suppose $G$ is embedded in $Y$. Given a perfect matching $\match$ for $G$, let $T_{\match}=\match\cap M_{G}$ and we call $T_{\match}$ a matching for $T_{G}$. There is a natural correspondence between 
\[s_{\match}:\SpinC(M_{G})\to \SpinC(Y)\]
defined as follows. Fix a non-zero vector field $v_{\match}$ on $T_{\match}$ so that it gives the orientation of $T_{\match}$. Any $\SpinC$ class $\spinc\in\SpinC(M_{G})$ may be represented by a nonzero vector field $v$ on $M_{G}$ such that $v|_{T_{\match}}=v_{\match}$. Then, we may modify $v$ in a small tubular neighborhood of $T_{\match}$ to construct a new vector field that extends to a non-zero vector field on $Y$. (See \cite[Section 8.1]{AE-2} for more details.) 

For instance, if $G$ is an embedded $\Theta_2$, we have two maps $s_1$ and $s_2$ corresponding to the matchings $\{e_1\}$ and $\{e_2\}$, respectively. Moreover, for any $\spinc\in\SpinC(M_G)$
 \[c_1(s_1(\spinc))-c_1(s_2(\spinc))=\mathrm{PD}[G]\]
where $[G]\in H_1(Y,\Z)$ is the homology class represented by the knot $G$ oriented as $(-e_1)\cup e_2$. Clearly, if $G$ is null-homologous, $s_{\match}$ doesn't depend on $\match$ and so we denote it by $s$. 

\begin{lem}\label{torsionspinc}
Assume $G$ is null-homologous. Then, every $\tsf$-torsion class $\spinc\in\SpinC(M_{G})$ corresponds to a torsion class $s(\spinc)\in\SpinC(Y)$.
\end{lem}

\begin{proof} By Lemma \ref{lem:tor-ind} it is enough to show this for $\tsf=\tsf_{\match}$ for some perfect matching $\match$.  For every $h\in H_2(Y)$ there exists a class $H\in H_2(M_{G})$ such that $i_{*} H=h$ and $[T_i]\cdot H=0$ for every $T_i\subset T_\match$. Here, $i:M_{G}\to M$ denotes the inclusion map. Therefore, for any $\tsf$-torsion class $\spinc\in \SpinC(M_{G})$ we have
\[\langle c_1(s(\spinc)),h\rangle=\langle c_1(\spinc),H\rangle=\sum_{T_i\subset \match} 2([T_i]\cdot H)=0.\]
\end{proof}


\subsubsection{Relative grading} Assume $(M,T)=(M_G,T_G)$ is a balanced tangle corresponding to an embedded graph $G\subset Y$ with $\mathsf{L}_G\neq\emptyset$. For any $\tsf\in\mathsf{L}_G$ and any $\tsf$-torsion $\SpinC$ class $\spinc\in\SpinC(M)$, we define a relative $\R$-grading on $\HFT(M,T,\spinc)$, denoted by $\grt$, as follows.

Consider an $\spinc$-admissible Heegaard diagram 
\[\HD=(\Sigma,\alphas,\betas,\z=\{z_1,...,z_{\el}\})\]
for $(M,T)$, where $z_i=T_i\cap \Sig$. Then, $\CFT(\HD,\spinc)$ is a chain complex with coefficients in $\Ring_G$ and we can define $\grt$ by setting:

\begin{equation}\label{Def:rel-grading}
\begin{cases}
\grt(\x)-\grt(\y)=\mu(\phi)-\sum_{j=1}^{\el}t_jn_{z_j}(\phi)\\
\grt(\la_i.\x)=\grt(\x)-t_i
\end{cases}
\end{equation}
where $\phi\in\pi_2(\x,\y)$.
\begin{lem}
With the above assumptions fixed, $\grt$ is well-defined.
\end{lem}

\begin{proof}
Let $\phi'\in\pi_2(\x,\y)$ be another Whitney disk from $\x$ to $\y$. Denote $\Pcal=\Dcal(\phi)-\Dcal(\phi')$, where $\Dcal(\phi)$ and $\Dcal(\phi')$ are the domains of $\phi$ and $\phi'$, respectively. Then, 
\[\mu(\phi)-\mu(\phi')=\langle c_1(\spinc),H(\Pcal)\rangle=\sum_{i=1}^{\el}t_i n_{z_i}(\Pcal)=\sum_{i=1}^{\el}t_i\left(n_{z_i}(\phi)-n_{z_i}(\phi')\right).\]
Here, $H(\Pcal)\in H_2(M_{G})$ denotes the homology class represented by $\Pcal$. Note that the second equality holds because $\spinc$ is  $\mathsf{t}$-torsion. Therefore, \[\mu(\phi)-\sum_{i=1}^{\el}t_in_{z_i}(\phi)=\mu(\phi')-\sum_{i=1}^{\el}t_in_{z_i}(\phi')\]
and so $\grt$ is well-defined.
\end{proof}

\begin{defn} Let $\Ring$ be an $\R_{\ge 0}$-graded $\F$-algebra. For any $\tsf\in\mathsf{L}_T$, a coloring $\cl:\pi_0(T)\to \Ring$ of $T$ is called \emph{$\tsf$-graded} if $\gr(\cl[T_i])=t_i$ for all $i=1,2,\cdots,\el$, where $\gr$ denotes the grading on $\Ring$.

%

%
\end{defn}

\begin{lem}
Suppose $\cl:\pi_0(T)\to\Ring$ is a $\tsf$-graded coloring on $(M,T)$, then $\grt$ induces a grading on $\HFT(\HD,\cl,\spinc)$, where $(\HD,\cl,\spinc)$ is a Heegaard diagram for the $\Ring$-tangle $\Tangle=[M,T,\cl,\spinc]$. Moreover, it induces a grading on $\HFT^{\Rin}(\HD,\cl,\spinc)$ for any graded $\Ring$-module $\Rin$.
\end{lem}
\begin{proof}
Note that $\HFT(\HD,\cl,\spinc)=\Ht_{*}\left(\CFT(\HD,\spinc)\otimes_{\phi}\Ring\right)$ and so the claim holds because $\cl$ is $\tsf$-graded. The second part is obvious.
\end{proof}

%
%
%

\begin{lem} For any $\tsf\in\mathsf{L}_T$, $\tsf$-torsion class $\spinc\in\SpinC(M)$ and $\tsf$-graded $\Ring$-coloring $\cl$ of $T$, $\grt$ induces a relative grading on $\HFT(M,T,\cl,\spinc)$ and so $\HFT^{\Rin}(M,T,\cl,\spinc)$ for any graded $\Ring$-module $\Rin$.
\end{lem}
\begin{proof}
Similar to the invariance proof of tangle Floer homology, see \cite{AE-1, AE-2}. \end{proof}

\subsubsection{Absolute grading} As before, $(M,T)$ is a balanced tangle corresponding to an embedded graph $G$. In this section, we assume $G$ is null-homologous, and for any $\tsf\in\mathsf{L}_G$ and any $\tsf$-torsion class $\spinc\in\SpinC(M)$, we lift $\grt$ to an absolute grading. Since $G$ is null-homologous and $\mathsf{L}_G\neq \emptyset$ we identify $\SpinC(M)$ with $\SpinC(Y)$, as before, and drop $s$ from the notation.

Suppose $\tsf=\tsf_{\match}$ for some perfect matching $\match$ for $G$. We lift the grading $\gr_{\match}=\gr_{\tsf_{\match}}$ to an absolute grading on $\CFT(M,T,\spinc)$ so that the induced grading on $\HFT^-(Y,\spinc)$ is the Maslov grading. More precisely, suppose $T_{\match}=\coprod_{i=1}^n T_{m_i}$, and consider the $\F[U_1,\cdots,U_n]$-coloring $\cl_{\match}$ of $T$ defined by setting $\cl_{\match}([T_j])=U_i$ if $j=m_i$ for some $1\le i\le n$, otherwise $\cl_{\match}([T_j])=1$. Then, $\cl_{\match}$ is $\tsf_{\match}$-graded and by definition 
\[\CFT(M,T,\cl_{\match},\spinc)\simeq \CFT^-(Y,\spinc).\]
Further, the relative $\Z$-grading $\grt$ on $\CFT(M,T,\cl_{\match},\spinc)$ coincides with the relative Maslov grading on $\CFT^-(Y,\spinc)$. Since,  $\spinc$ is torsion, relative Maslov grading on $\CFT^-(Y,\spinc)$ lifts to an absolute grading. Consequently, we get a lift of $\gr_{\match}$to an absolute grading on $\CFT(M,T,\cl_{\match},\spinc)$, which will give a lift of $\gr_{\match}$ to an absolute $\Z$-grading on $\CFT(M,T,\spinc)$.

Suppose $\tsf=a_1\tsf_{\match_1}+\ldots+a_k\tsf_{\match_k}$ where $a_i>0$ and $\match_i$ is a perfect matching for every $i$. It is easy to see that for any $\x,\y\in\Ta\cap \Tb$
\[\grt(\x)-\grt(\y)=\left(a_1\gr_{\match_1}(\x)+\ldots+a_k\gr_{\match_k}(\x)\right)-\left(a_1\gr_{\match_1}(\y)+\ldots+a_k\gr_{\match_k}(\y)\right).\]

We inductively put a $\Delta$-complex structure on $\mathsf{L}_G$. If $\dim\mathsf{L}_G=1$, then $\mathsf{L}_G$ is a one-simplex with two vertices, so a $\Delta$-complex. Equip finite set of weights corresponding to the perfect matchings i.e. the vertices of the polygon $\mathsf{L}_G$ or
\[\{\tsf_{\match}\ |\ \match:\ \text{perfect matching for}\ G\}\] with lexicographic ordering. Let $\tsf_{\min}$ be the smallest point.
If $\dim \mathsf{L}_G=d$, boundary of $\mathsf{L}_G$ is a union of polygons of dimension $d-1$. Let $\del_0\mathsf{L}_G$ be the union of boundary polygons that do not contain $\tsf_{\min}$. By induction, we put a $\Delta$-complex structure on every polygon in $\del_0\mathsf{L}_G$, and define a $\Delta$-complex structure on $\mathsf{L}_G$ so that its $d$-simplicies are the cone of the $(d-1)$-simplicies in $\del_0\mathsf{L}_G$ at $\tsf_{\min}$. 
%

For a given $\tsf\in L_G$, consider the $i$-simplex $\Delta^i$ so that $\tsf\in\mathrm{int}(\Delta^i)$. Assume $\Delta^i$ is given by the convex hull of $\{\tsf_{\match_1},\cdots,\tsf_{\match_{i+1}}\}$. We lift $\grt$ to an absolute grading by setting 
\[\grt(\x)=n_1\gr_{\match_1}(\x)+\ldots+n_i\gr_{\match_{i+1}}(\x)\]
where $\tsf=n_1\tsf_{\match_1}+\ldots+n_{i+1}\tsf_{\match_{i+1}}$.

Consequently, for any $\tsf$-graded $\Ring$-coloring $\cl$ on $(M,T)$, and any graded $\Ring$-module $\Rin$, we get an absolute grading $\grt$ on $\HFT^{\Rin}(M,T,\cl,\spinc)$.

\begin{example} \label{ex:link} Suppose that the graph $L=\coprod_{i=1}^{n}L_i$ is an $n$-component link embedded in $S^3$ such that each connected component $L_i$ of $L$ contains exactly two vertices. Then, we label the edges by $1,2, \ldots, 2n$ such that $L_i=e_i\cup e_{n+i}$. Note that this labeling specifies an orientation on $L$ by assuming $L_i=e_i\cup (-e_{n+i})$. Any Heegaard diagram $\HD=(\Sig,\alphas,\betas,\{z_1,z_2,\ldots,z_{2n}\})$ for the corresponding tangle, will give a pointed Heegaard diagram in the sense of \cite[Definition 3.7]{OS-linkinvariant} for $L$ by calling each base point $z_{n+i}$, $w_{i}$.

Suppose $\match_{\w}=\coprod_{i=1}^n e_{i+n}$ and $\match_{\z}=\coprod_{i=1}^n e_{i}$. Then, $\tsf_{\min}=\tsf_{\match_{\w}}$ and $\tsf_{\max}=\tsf_{\match_{\z}}$.  Note that $\mathsf{L}_L$ is an $n$-dimensional cube in $\R^{2n}$ defined by the linear equations 
\[\left\{t_i+t_{n+i}=2\ | \ 1\le i\le n\ \ \text{and}\ \ t_i\in[0,2]\right\}.\]
Further, boundary cubes in $\del\mathsf{L}_L$ that are not adjacent to $\tsf_{\min}$ contain $\tsf_{\max}$ as a $0$-simplex. Thus, the diagonal connecting $\tsf_{\min}$ to $\tsf_{\max}$ is a $1$-simplex in the $\Delta$-complex structure on $\mathsf{L}_L$. For any $t\in [0,2]$, let
\[\tsf=\frac{1}{2}\left(t \tsf_{\max}+(2-t)\tsf_{\min}\right).\]  
By definition, $\grt(\x)=\frac{1}{2}(t \gr_{\match_\z}(\x)+(2-t)\gr_{\match_\w}(\x))$. On the other hand, $\gr_{\match_{\w}}=\gr_{\w}$ and $\gr_{\match_{\z}}=\gr_{\z}$, defined in \cite{Zemke-gr}. Therefore, for any $\x\in\Ta\cap\Tb$
\[\gr_{\match_{\w}}(\x)=M(\x),\quad\quad\text{and}\quad\quad A(\x)=\frac{\gr_{\match_{\w}}(\x)-\gr_{\match_\z}(\x)}{2},\]
where $M$ and $A$ denote the Maslov and Alexander gradings, respectively. Thus, $\grt(\x)=M(\x)-tA(\x)$ which is equal to $\gr_t(\x)$ defined in \cite{OSS-upsilon}.

%
\end{example}
%

%

\section{The $\tvec$-modified Chain Complex and Upsilon}\label{sec:upsilon}
Assume the graph $G\subset Y$ is null-homologous, and $(M_{G},T_{G})$ is the corresponding tangle. We will define the $\tsf$-modified chain complex $\tsf\CFT(M_G,T_G,\spinc)$ for any $\tsf\in\mathsf{L}_G$ and any $\tsf$-torsion class $\spinc\in\SpinC(M_G)$ as follows.

Let $\Rring$ be the ring of long power series defined in \cite{OSS-upsilon}. Specifically, it is the group of formal sums
\[\left\{\sum_{t\in A}\la^t\ |\ A\subset\R_{\ge 0},\ A\ \text{well-ordered}\right\},\]
where $\R_{\ge 0}$ is the set of nonnegative real numbers, and the order on $A$ is induced from $\R$. Moreover, the multiplication in $\Rring$ is given by 
\[\left(\sum_{t\in A}\la^t\right)\cdot\left(\sum_{s\in B}\la^{s}\right)=\sum_{r\in A+B}\#\{(t,s)\in A\times B\ |\ t+s=r\}\cdot \la^r\]
where $A+B$ is the image of the summation map from $A\times B$ to $\R_{\ge 0}$, and the coefficient of $\la^r$ is the number modulo two. Note that $A+B$ is well-ordered. We may assume $\Rring$ is $\R_{\ge 0}$-graded such that the grading of $\la^t$ is $t$. 

To any $\tsf\in\mathsf{L}_G$ we assign an $\Rring$-coloring $\cl_{\tsf}$ of $(M_{G},T_{G})$ by setting $\cl_{\tsf}(T_i)=\la^{t_i}$. We define 
\[\tsf\CFT(M_{G},T_{G},\spinc):=\CFT(M_{G},T_{G},\spinc,\cl_{\tsf}).\]
Considering $\grt$ on $\CFT(M_{G},T_{G},\spinc)$, the coloring $\cl_{\tsf}$ is $\tsf$-graded, so $\grt$ induces a grading on $\tsf\CFT(M_G,T_G,\spinc)$, still denoted by $\grt$.


\begin{lem} \label{lem:infty1}
For any $\tsf\in\mathsf{L}_G$ 
\[\tsf\CFT(M_{G},T_{G},\spinc)\otimes \Rring^{*}\simeq\tsf_{\match}\CFT(M_{G},T_{G},\spinc)\otimes\Rring^{*}\] where $\match$ is any perfect matching and $\spinc$ is a torsion class. Here, $\Rring^*$ is the ring of fractions of $\Rring$. 

\end{lem}
\begin{proof}
Let \[\HD=(\Sig,\alphas,\betas,\cl_{\tsf}:\z\to \Rring,\spinc)\] be a Heegaard diagram for the $\Rring$-tangle $\Tangle=[M_G,T_G,\cl_{\tsf},\spinc]$, where $\z=\coprod_{i=1}^{\el}z_i$. Further, assume $\z_{\match}\subset\z$ denotes the base points corresponding to $\match$. For simplicity, suppose $\z_{\match}=\coprod_{i=1}^nz_i$. Replacing $\cl_{\tsf}$ with $\cl_{\tsf_{\match}}$ we get a diagram for $[M_{G},T_G,\cl_{\tsf_{\match}},\spinc]$, denoted by $\HD_{\match}$. 

Pick a relative $\SpinC$ class $\relspinc_0$ and an intersection point $\x_0\in\Ta\cap \Tb$ such that $\relspinc(\x_0)=\relspinc_0$ and $[\relspinc_0]=\spinc$. Let $f:\CFT(\HD)\otimes\mathcal{R}^*\to\CFT(\HD_{\match})\otimes \mathcal{R}^*$ be the map defined by 
\[f(\x)=\la^{-\sum_{i=1}^{\el}t_in_{z_i}(\phi)}\x\]
 where $\phi\in\pi_2(\x_0,\x)$ so that $n_{z_i}(\phi)=0$ for $i=1,\ldots, n$. First, we show that $f$ is well-defined. Let $\phi'\in\pi_2(\x_0,\x)$ be another Whitney disk with $n_{z_i}(\phi)=0$ for any $i=1,\ldots, n$. Then, $\Pcal=\phi'*\phi^{-1}$ is a periodic domain, and by the proof of Lemma \ref{lem:tor-ind} \[\left(\sum_{i=1}^{\el}t_in_{z_i}(\phi')\right)-\left(\sum_{i=1}^{\el}t_in_{z_i}(\phi)\right)=\sum_{i=1}^{\el}t_in_{z_i}(\Pcal)=\sum_{i=1}^{n}2n_{z_i}(\Pcal)=0.\]
Note that if $\relspinc(\x)=\relspinc_0$, then there is a disk $\phi\in\pi_2(\x_0,\x)$ such that $n_{z_i}(\phi)=0$ for all $i=1,\ldots,\el$. Therefore, $f(\x)=\x$. This observation also implies that $f$ does not depend on the choice of $\x_0$. 

Next, we show that $f$ is a chain map. Assume $\phi\in\pi_2(\x_0,\x)$ and $\phi'\in\pi_2(\x,\y)$ are such that $n_{z_i}(\phi)=0$ for all $i=1,\ldots, n$ and $\phi'$ contributes to $\del \x$. Then,
\[\del\circ f(\x)=\la^{-\left(\sum_{i=1}^{\el}t_in_{z_i}(\phi)\right)+\left(\sum_{i=1}^n2n_{z_i}(\phi')\right)}\x\]

Let $\Sig\setminus\alphas=\coprod_{i=1}^nA_i$ such that $z_i\in A_i$. Consider the Whitney disk $\phi''\in\pi_2(\x_0,\y)$ so that $\Dcal(\phi'')=\Dcal(\phi)+\Dcal(\phi')-\sum_{i=1}^nn_{z_i}(\phi')A_i$. Then, $n_{z_i}(\phi'')=0$ for all $i=1,\ldots,n$ and so \[f(\y)=\la^{-\left(\sum_{i=1}^{\el}t_i(n_{z_i}(\phi)+n_{z_i}(\phi'))\right)+\left(\sum_{i=1}^{n}2n_{z_i}(\phi')\right)}\y.\]
Thus, $f\circ\del (\x)=\del\circ f(\x)$ and $f$ is a chain map. Since, we are working over $\Rring^{*}$ it is straightforward that $f$ is a chain homotopy equivalence, and so we are done.


\end{proof}

\begin{lem}\label{lem:infty2} If $Y$ is a rational homology sphere, then 
\[H_{*}\left(\tsf_{\match}\CFT(M_G,T_G,\spinc)\otimes\Rring^{*}\right)\cong \left(\Rring^{*}_{-\frac{1}{2}}\oplus\Rring^{*}_{\frac{1}{2}}\right)^{n-1}\] 
for any perfect matching $\match$ and $\SpinC$ class $\spinc\in\SpinC(M_{G})$.

\end{lem}
\begin{proof}
See \cite[Theorem 4.4]{OS-linkinvariant}, \cite[Theorem 10.1]{OS-3m2} and \cite[Lemma 10.1]{OSS-upsilon}.
\end{proof}

We conclude this Section by defining the upsilon invariant. Lemmas \ref{lem:infty1} and \ref{lem:infty2} imply that $\tsf\HFT(M_{G},T_{G},\spinc)/\mathrm{Tors}$ is a free $\Rring$-module of rank $2^{n-1}$. So, we define the upsilon invariant for graphs embedded in rational homology spheres as follows. 

\begin{defn}
For any $\tsf\in\mathsf{L}_{G}$, choose a basis $\{e_i(\tsf)\}_{i=1}^{2^{n-1}}$ of homogeneous elements for $\tsf\HFT(M_{G},T_{G},\spinc)/\mathrm{Tors}$, so that $\grt(e_i(\tsf))\le\grt(e_{i+1}(\tsf))$ for any $1\le i\le 2^{n-1}-1$.
Then, $\Upsilon_{G,\spinc}(\tsf)$ is an ordered sequence of $2^{n-1}$ real numbers defined as 
\[\Upsilon_{G,\spinc}(\tsf):\left\{\Upsilon_{\min}=\Upsilon_1\le\Upsilon_2\le\ldots\le\Upsilon_{2^{n-1}}=\Upsilon_{\max}\right\}\]
where $\Upsilon_i=\grt(e_i(\tsf))$.
\end{defn}

Note that $\Upsilon_{G,\spinc}(\tsf)$ is an invariant of the matched graph $G$ and the $\SpinC$ class $\spinc\in\SpinC(M_{G})=\SpinC(Y)$.

\begin{defn}\label{def:vsum}
Suppose $G_1$ and $G_2$ are $\Theta_n$-type graphs embedded in rational homology spheres $Y_1$ and $Y_2$, respectively. An embedded graph $G$ in $Y_1\#Y_2$ is obtained by \emph{vertex connected sum} of $G_1$ with $G_2$ if it is obtained by removing small ball neighborhoods $N_+\subset Y_1$ and $N_-\subset Y_2$ around $v_+(G_1)$ and $v_-(G_2)$, respectively, and identifying $\del N_+$ with $\del N_-$ with an orientation reversing diffeomorphism that maps $e(G_1)\cap\del N_+$ to $e(G_2)\cap \del N_-$ and respects the indices of edges. 
\end{defn}

%

\begin{proof} (Lemma \ref{lem:sum}) Suppose $(M^1,T^1)$ and $(M^2,T^2)$ are tangles corresponding $G_1$ and $G_2$, respectively. Since $G$ is a vertex connected sum of $G_1$ and $G_2$, the tangle $(M,T)$ corresponding to $G$ is obtained by concatenating $(M^1,T^1)$ and $(M^2,T^2)$ with some diffeomorphism $d$ from $(\del_+M^1,\del_+T^1)$ to $(\del_-M^2,\del_-T^2)$ that respects the indices of $T^1=\coprod_{i=1}^nT^1_i$ and $T^2=\coprod_{i=1}^nT^2_i$. The proof of Theorem \ref{TFH-concatenation} implies that 
\[\CFT(M,T,\spinc)\simeq\CFT(M^1,T^1,\spinc_1)\otimes \CFT(M^2,T^2,\spinc_2),\]
and it is easy to check that it preserves relative $\tsf$-grading for any $\tsf$. Moreover, any $1\le i\le n$ specifies a perfect matching for $T^1$, $T^2$ and $T$, denote the corresponding gradings by $\gr_i$. By \cite[Section 4]{OS-d-invariant} the above chain homotopy preserves $\gr_i$ for all $i$, and thus by definition $\grt$ for all $\tsf$.

For $i=1,2$, let $\cl^i_{\tsf}$ be the $\mathcal{R}$-coloring on $(M^{i},T^{i})$ associated with $\tsf$ i.e. $\cl^{i}_{\tsf}([T^{i}_j])=\la^{t_j}$. Denote $\Tangle^{i}_{\tsf}=[M^{i},T^{i},\cl^{i}_{\tsf},\spinc_i]$ for $i=1,2$. Consequently, we get a grading preserving isomorphism
\[\HFT(\Tangle)\cong H_{*}\left(\CFT(\Tangle^1)\otimes_{\mathcal{R}}\CFT(\Tangle^2)\right)\]
for $\Tangle=\Tangle^1\circ_d\Tangle^2$ and the claim follows from the K\"{u}nneth formula. 
\end{proof}

\begin{defn}\label{def:homcob}
Suppose $Y$ and $Y'$ are rational homology spheres, and $G\subset Y$ and $G'\subset Y'$ are embedded graphs. We say $G$ and $G'$ are \emph{homology concordant} if there exists a homology cobordism $(W,F)$ from $(M_G,T_G)$ to $(M_{G'},T_{G'})$ and $F$ preserves the indicies. Precisely, if $T_G=\coprod_{i=1}^\el T_i$ and $T_{G'}=\coprod_{i=1}^{\el} T'_i$ with indices inherited from $e(G)$ and $e(G')$, respectively, then $T_i$ and $T'_i$ lie on the boundary of the same connected component of $F$ for every $i$. 

Suppose $G$ and $G'$ are null-homologous, and $\mathsf{L}_G,\mathsf{L}_{G'}\neq\emptyset$. Given $\spinc\in\SpinC(Y)$ and $\spinc'\in\SpinC(Y)$ we say $(G,\spinc)$ and $(G',\spinc')$ are $\SpinC$ \emph{homology cobordant} if they are homology concordant and the homology cobordism can be equipped with a $\SpinC$ class $\spinct\in\SpinC(W)$ so that $\spinct|_{M_G}=\spinc$ and $\spinct|_{M_{G'}}=\spinc'$.
\end{defn}

%

\begin{proof} (Theorem \ref{thm:invhomcob}) Assume $(M,T)$ and $(M',T')$ denote the tangles corresponding to $G$ and $G'$, and $(W,F)$ is the homology cobordism from $(M,T)$ to $(M',T')$. Let $\Ring=\Ring_T=\Ring_{T'}$, and $\cl_F$ be the $\Ring$-coloring on $F$ defined by $\cl_F([F_i])=\la_i$. Then, the $\Ring$-cobordism $\Cob=[W,F,\cl_F,\spinct]$ will induce a chain map $\fmap_{\Cob}$ from $\CFT(M,T,\spinc)$ to $\CFT(M',T',\spinc')$. 
Since $(W,F)$ is a homology cobordism, $F$ gives a one-on-one correspondence between the perfect matchings for $G$ and $G'$. For any perfect matching $\match\subset G$ we denote the corresponding perfect matching for $G'$ by $\match$ as well. If $T_{\match}=\coprod_{i=1}^nT_{j_i}$, let $F_{\match}=\coprod_{i=1}^nF_{j_i}$ and $\cl_{\match}$ be the $\F[U_1,\cdots,U_n]$-coloring of $F$ defined by setting $\cl_{\match}([F_{j_i}])=U_i$ while $\cl_{\match}([F_j])=1$ for all $j\neq j_1,\cdots,j_n$. Replacing $\cl_F$ with $\cl_{\match}$ in $\Cob$ we get a cobordism denoted by $\Cob_{\match}$, and let $\fmap_{\Cob}^{\match}$ be the corresponding cobordism map. As discussed in Section \ref{sec:background}, under the identifications
\[\CFT(M,T,\cl_{\match},\spinc)=\CFT^-(Y,\spinc)\quad\text{and}\quad\CFT(M',T',\cl_{\match},\spinc')=\CFT^-(Y',\spinc'),\] 
 $\fmap_{\Cob}^{\match}$ is the same as the multi-pointed version of Ozsv\'{a}th-Szab\'{o}'s cobordism maps corresponding to the cobordism $Z$ from $Y$ to $Y'$, equipped with the $\SpinC$ class corresponding to $\spinct$. Therefore,  \cite[Theorem 7.1]{OS-4m} and \cite[Theorem 1.4]{Zemke-gr} imply that $\fmap_{\Cob}^\match$ preserves the Maslov gradings and so $\fmap_{\Cob}$ preserves $\tsf_{\match}$-gradings, denoted by $\gr_{\match}$. Consequently, for any $\tsf=(t_1,\cdots,t_{\el})\in\mathsf{L}_T=\mathsf{L}_{T'}$, the chain map $\fmap_{\Cob}$ preserves $\tsf$-grading. So, equipping $(W,F)$ with the $\Rring$-coloring $\cl_{\tsf}$ defined by $\cl_{\tsf}([F_i])=\la^{t_i}$ we get a grading preserving cobordism map denoted by $\fmap_{\Cob}^{\tsf}$ from $\tsf\CFT(M,T,\spinc)$ to $\tsf\CFT(M',T',\spinc')$.

Next, pick relative $\SpinC$ structures $\relspinc_0\in\SpinC(M,T)$ and $\relspinc_0'\in\SpinC(M',T')$ so that there exists a relative $\SpinC$ class $\relspinct\in\SpinC(W,F)$ with $\relspinct|_{(M,T)}=\relspinc_0$ and $\relspinct|_{(M',T')}=\relspinc_0'$. We show that the diagram 
\begin{diagram}
\tsf\CFT(M,T,\spinc)\otimes\Rring^{*}&&\rTo{\fmap_{\Cob}^{\tsf}}&&\tsf\CFT(M',T',\spinc')\otimes\Rring^*\\
\dTo{f}&&&&\dTo{f'}\\
\tsf_{\match}\CFT(M,T,\spinc)\otimes\Rring^{*}&&\rTo{\fmap_{\Cob}^{\tsf_{\match}}}&&\tsf_{\match}\CFT(M',T',\spinc')\otimes\Rring^*\\
\end{diagram} 
commutes up to chain homotopy, where $f$ and $f'$ are the chain homotopies defined in Lemma \ref{lem:infty1} using $\relspinc_0$ and $\relspinc_0'$, respectively.  By definition, $\fmap_{\Cob}^{\square}=\fmap_{\Cob_1}^{\square}\circ\fmap_{\Cob_2}^{\square}\circ\fmap_{\Cob_3}^{\square}$ where $\Cob=\Cob_1\circ\Cob_2\circ \Cob_3$ and $\Cob_i$ is defined by attaching $i$-handles. Here $\square=\tsf,\tsf_{\match}$. It is easy to check that the $1$- and $3$-handle attachment cobordism maps commute with the corresponding vertical chain homotopies. Thus, it is enough to prove this commutativity for the case that $\Cob$ corresponds to attaching $2$-handles and $(W,F)$ is parametrized by a framed link $\mathbb{L}$ in $M$.  Consider a  Heegaard triple $\HD=(\Sigma,\alphas,\betas,\gammas,\z)$ adapted to the link, and assume $(\HD,\cl_{\tsf},\spinct)$ and $(\HD,\cl_{\tsf_{\match}},\spinct)$ are $\Rring$-diagrams i.e. $\spinct$-admissible.  

Fix $\x_0\in \Ta\cap\Tb$ and $\y_0\in\Ta\cap\Tc$ so that $\relspinc(\x_0)=\relspinc_0$ and $\relspinc(\y_0)=\relspinc_0'$. Let $\psi_0\in\pi_2(\x_0,\Theta_{\beta\gamma},\y_0)$ be a triangle representing $\relspinct$ with $\mu(\psi_0)=0$ and $n_{z_i}(\psi_0)=0$ for all $1\le i\le\el$. For simplicity, assume the subset of base points specified by the matching is given by $\z_{\match}=\coprod_{i=1}^nz_i$. Let $\psi\in\pi_2(\x,\Theta_{\beta\gamma},\y)$ be a triangle contributing to $\fmap_{\Cob}^{\tsf}$ and so $\fmap_{\Cob}^{\tsf_{\match}}$. By definition, 
\[f(\x)=\la^{-\sum_{i=1}^{\el}t_in_{z_i}(\phi)}\x\quad\quad\text{and}\quad\quad f'(\y)=\la^{-\sum_{i=1}^{\el}t_in_{z_i}(\phi')}\y\]
where $\phi\in\pi_2(\x_0,\x)$ and $\phi'\in\pi_2(\y_0,\y)$ are Whitney disks with $n_{z_i}(\phi)=n_{z_{i}}(\phi')=0$ for all $1\le i\le n$. Then, $\psi$ contributes to $f'\circ \fmap_{\Cob}^{\tsf}(\x)$ and $\fmap_{\Cob}^{\tsf_{\match}}\circ f(\x)$ as 
\[\#\mathcal{M}(\psi)\la^{\sum_{i=1}^{\el}t_in_{z_i}(\psi)-\sum_{i=1}^{\el}t_in_{z_i}(\phi')}\y\quad\text{and}\quad\#\mathcal{M}(\psi)\la^{\sum_{i=1}^{n}2n_{z_i}(\psi)-\sum_{i=1}^{\el}t_in_{z_i}(\phi)}\y\]
respectively.  On the other hand, triangles $\phi*\psi*(\phi')^{-1}$ and $\phi_0$ differ in doubly periodic domains and so
\[\sum_{i=1}^{\el}t_i(n_{z_i}(\phi)+n_{z_i}(\psi)-n_{z_i}(\phi'))=\sum_{i=1}^{n}2(n_{z_i}(\phi)+n_{z_i}(\psi)-n_{z_i}(\phi'))=\sum_{i=1}^n2n_{z_i}(\psi).\]
Thus, $\psi$ contributes equally to $f'\circ \fmap_{\Cob}^{\tsf}(\x)$ and $\fmap_{\Cob}^{\tsf_{\match}}\circ f(\x)$, and so the above diagram commutes. 
 
 Since, $\left(\fmap_{\Cob}^{\tsf_{\match}}\right)_*$ in the above diagram is an isomorphism, 
 \[\left(\fmap_{\Cob}^{\tsf}\right)_*: H_*\left(\tsf\CFT(M,T,\spinc)\otimes\Rring^{*}\right)\to H_*\left(\tsf\CFT(M',T',\spinc')\otimes\Rring^*\right)\] is also an isomorphism. Thus, the induced map 
 \[F^{\tsf}:H_*(\tsf\CFT(M,T,\spinc))/\text{Tor}\to H_*(\tsf\CFT(M',T',\spinc'))/\text{Tor}\]
by $\fmap_{\Cob}^{\tsf}$ is injective. Further, $F^{\tsf}$ is grading preserving. For any $r\in\R$, let $m_r$ and $m'_r$ be the smallest $i$ such that $\Upsilon_i(\tsf)>r$ to $G$ and $G'$, respectively. Then, $m'_r\ge m_r$. Similarly, reversing the orientation on $W$ and $F$, we get a homology cobordism from $(M',T')$ to $(M,T)$, and repeating the above argument for this cobordism implies that $m_r\ge m'_r$. Therefore, $m_r=m'_r$ and so $\Upsilon_{G,\spinc}(\tsf)=\Upsilon_{G',\spinc'}(\tsf)$.
\end{proof}

\begin{example} \label{ex:closed}
Suppose $G$ is an embedded $\Theta_1$-graph in a rational homology sphere $Y$ i.e. $G$ has two vertices and one edge connecting them. Then, $(M_G,T_G)$ is the tangle corresponding to $Y$ with one marked point, and $\mathsf{L_{G}}=\{2\}\in\R$. Thus, for the only solution $\tsf=2$
\[\tsf\CFT(M_G,T_G,\spinc)=\CFT^-(Y,\spinc)\otimes_{\F[U]} \Rring\]
where $\Rring$ is an $\F[U]$-module using the homomorphism $\phi$ defined by setting $\phi(U)=\la^2$. Moreover, $\tsf$-grading on the left chain complex is the grading induced by the Maslov grading on the right. Since $\Rring$ is flat, we have
\[\Upsilon_{G,\spinc}(2)=d(Y,\spinc)\]
the \emph{d-invariant} or \emph{correction term} of $Y$ \cite{OS-d-invariant}.

\end{example}

\begin{example}
Suppose the graph $L=\coprod_{i=1}^nL_i$ is an $n$-component link (knot if $n=1$) embedded in $S^3$ such that each connected component $L_i$ of $L$ contains exactly two vertices and two edges, and edges are labelled by $i$ and  $n+i$. Then, it is straightforward from the definition and Example \ref{ex:link} that for any $t\in [0,2]$
\[\Upsilon_{L}(t,t,\cdots,t,2-t,2-t,\cdots,2-t)=\Upsilon_{\L}(t)\]
where on the left hand side the number of entries equal to $t$ and $2-t$ is $n$. Moreover, $\L$ denotes the corresponding oriented link and $\Upsilon$ on the right hand side is the original link invariant defined in \cite{OSS-upsilon}.
 \end{example}

 \begin{lem}\label{lem:PL} For any embedded graph $G\subset M$ in a rational homology sphere, $\Upsilon_{\spinc,G}(\tsf)$ is a sequence of $2^{n-1}$ continuous piecewise linear functions over $\tsf\in\mathsf{L}_G$.
 \end{lem}
 \begin{proof} It is similar to the proof of \cite[Proposition 1.4]{OSS-upsilon}. Note that $\tsf\CFT(M_G,T_G,\spinc)$ is a continuously varying family of finitely generated chain complexes over $\mathcal{R}$ indexed by $\tsf\in\mathsf{L}_G$ in the sense of \cite[Definition 5.1]{OSS-upsilon}. One difference is that $H_{*}\left(\tsf\CFT(M_G,T_G,\spinc)\right)$ has rank $2^{n-1}$. But still an analogous argument as in  \cite[Proposition 5.2]{OSS-upsilon} implies $\Upsilon_i(\tsf)$ in the sequence $\Upsilon_{\spinc,G}(\tsf)$ is a continuous function of $\tsf$. Moreover, $\Upsilon_i(\tsf)$ is equal to the $\grt$ for a generator of $\tsf\CFT(M_G,T_G,\spinc)$ i.e. the intersection points in $\Ta\cap\Tb$. The number of generators are finite and $\grt(\x)$ for any $\x\in\Ta\cap\Tb$ is a linear function. So, $\Upsilon_i$ is piecewise linear.   \end{proof}
 
Next, we consider the special class of $\Theta_n$-type graphs embedded in $S^3$, and discuss the behavior of $\Upsilon$ and its directional derivatives at the vertices of $\mathsf{L}_G$.

 \begin{lem} Suppose $G$ is a $\Theta_n$-type graph embedded in $S^3$. Then, the followings hold:
 \begin{enumerate}
 \item Let $\tsf^i\in\mathsf{L}_G$ be the solution corresponding to the perfect matching $e_i$ i.e. $t_j=0$ for all $j\neq i$ while $t_i=2$. Then, $\Upsilon_{G}(\tsf^i)=0$.
 \item For any $i\neq j$, let $V^i_j=(v_1,\cdots, v_n)$ where $v_i=-1$, $v_j=1$ and $v_k=0$ for $k\neq i,j$. Then, $D_{V^i_j}\Upsilon_G(\tsf^i)$ is equal to $-\tau(K_{ij})$ where $K_{ij}=e_j\cup(-e_i)$. 
 \end{enumerate}
 \end{lem}
 
\begin{proof}
The proof of the first part is similar to \cite[Proposition 1.5]{OSS-upsilon}. The second part follows from \cite[Proposition 1.6]{OSS-upsilon} and the fact that for every $\tsf$ with $t_k=0$ when $k\neq i,j$  \[\Upsilon_{G}(\tsf)=\Upsilon_{K_{ij}}(t_i).\]
\end{proof} 



\section{Application: Cobordism group of homology cylinders} \label{sec:app}
Over a surface $S_{g,n}$ of genus $g$ with $n$ boundary components, Garoufalidis and Levine introduce homology cobordism group of homology cylinders, in \cite{GL,Le}. In this Section, we use $\Upsilon$ to define a family of homomorphisms on this group, when $g=0$.

A \emph{homology cylinder} over $S=S_{g,n}$ is a triple $(X,i_+,i_-)$ where $X$ is a three-manifold with boundary and orientation preserving (resp. orientation reversing) embeddings $i_{+}$ (resp. $i_-$) from $S_{g,n}$ into $\del X$ such that
\begin{enumerate}
\item $\del X=i_+(S)\cup i_-(S)$ and $i_+(S)\cap i_-(S)=i_+(\del S)=i_-(\del S)$.
\item $i_+|_{\del S}=i_-|_{\del S}$.
\item $i_{\pm}$ induce isomorphisms from $H_{*}(S)$ to $H_{*}
(X)$.
 \end{enumerate}
 In other words, $(X,i_+,i_-)$ is a homology cobordism from $S$ to itself. 
 
 Two homology cylinders $(X,i_+,i_-)$ and $(X',i_+',i_-')$ over $S$ are called \emph{equivalent} if there exists a diffeomorphism $f:X\to X'$ such that $i_{\pm}'=f\circ i_{\pm}$.  For any $g,n\ge 0$, the equivalence class of homology cylinders over $S_{g,n}$ along with the \emph{stacking} operation defined as
 \[(X,i_+,i_-)\cdot(X',i_+',i_-')=\left(X\cup_{i_-'\circ (i_+)^{-1}}X',i_-,i_+'\right)\]
 form a monoid, denoted by $\mathcal{C}_{g,n}$.
 
 Two homology cylinders $(X,i_+,i_-)$ and $(X',i_+',i_-')$ over $S$ are called smoothly (resp. topologically) \emph{homology cobordant} if there exists a smooth (resp. topological) four-manifold $Z$ with boundary
 \[\del Z=\frac{-X\cup X'}{\left\langle i_{\pm}(x)=i_{\pm}'(x)\ |\ x\in S \right\rangle}\] 
 such that inclusions of $X,X'$ in $Z$ induce isomorphisms from $H_\ast(X)$ and $H_{\ast}(X')$ to $H_{\ast}(Z)$. Under these conditions, $Z$ is called a \emph{smooth (resp. topological) homology cobordism} from $(X,i_+,i_-)$ to $(X',i_+',i_-')$. Being smooth (resp. topological) homology cobordant is an equivalence relation, and the quotient of $\mathcal{C}_{g,n}$ with this relation is a group denoted by $\mathcal{H}_{g,n}^{\smooth}$ (resp. $\mathcal{H}_{g,n}^{\topo}$), called \emph{smooth (resp. topological) homology cobordism group of homology cylinders over $S_{g,n}$}. See \cite{Le} for more details.
 
 \begin{convention}
  From now on, whenever we do not specify smooth or topological we mean smooth. 
 \end{convention}
  
 \begin{example} For any diffeomorphism $\phi\in \mathcal{M}_{g,n}$, we have a homology cylinder 
 \[X_\phi=\left(\frac{S_{g,n}\times [0,1]}{\left\langle (x,t)\sim (x,s)\ |\ x\in \del S_{g,n}\ \text{and}\ t,s\in [0,1]\right\rangle}, \mathrm{id}\times\{0\},\phi\times\{1\}\right).\]
 Mapping $\phi$ to $X_{\phi}$ defines a homomorphism from $\mathcal{M}_{g,n}$ to $\mathcal{H}_{g,n}$. Garoufalidis-Levine prove that this homomorphism is injective for $n=1$ (See \cite[Section 2.4]{GL}) and Cha-Friedl-Kim extende their work for any $n\ge 0$ (See \cite[Proposition 2.4]{cha_friedl_kim_2011}).
 \end{example}

For the rest of this section, assume $g=0$ and $S=S_{0,n}$. Let $(X,i_+,i_-)$ be a homology cylinder over $S$ and $\tau=i_+(\del S)$. Then, $(X,\tau)$ is a sutured manifold. The tangle corresponding to $(X,\tau)$ is associated to an embedding $G$ of $\Theta_n$ in an integral homology sphere obtained from filling in the sutures and then attaching three-handles over the two resulted sphere boundary components. 
%
%
%
%
On the other hand, for any sutured manifold $(X,\tau)$ corresponding to an embedding of $\Theta_n$ in an integral homology sphere we have $R_+(\tau)\cong R_-(\tau)\cong S$. Equipping $(X,\tau)$ with diffeomorphisms $i_{\pm}:S\to R_{\pm}(\tau)$ so that $i_+$ is orientation preserving, $i_-$ is orientation reversing and $i_+|_{\del S}=i_-|_{\del S}$ we get a homology cylinder  $(X,i_+,i_-)$ over $S$.

Fix a labeling of $\del S=\coprod_{i=1}^n\del_iS$. For any homology cylinder $(X,i_+,i_-)$ over $S$ and any vector $\tsf=(t_1,\ldots,t_n)\in [0,2]^n$ with $t_1+\ldots+t_n=2$, let
\[\Upsilon_{X}(\tsf):=\Upsilon_{G}(\tsf),\] 
where $G$ denotes the embedded graph. Since $\Upsilon_{G}(\tsf)$ does not depend on $i_+$ and $i_-$ we drop them from the notation. 
\begin{lem}\label{lem:coninv}
If homology cylinders $(X,i_+,i_-)$ and $(X',i_+',i_-')$ over $S=S_{0,n}$ are homology cobordant, then
\[\Upsilon_{X}(\tsf)=\Upsilon_{X'}(\tsf)\]
for every $\tsf\in [0,2]^n$ with $t_1+\ldots+t_n=2$.
\end{lem}

\begin{proof}
Let $(M,T)$ and $(M',T')$ be the tangles corresponding to $(X,i_+,i_-)$ and $(X',i'_+,i'_-)$, respectively, and $Z$ be a smooth homology cobordism from $(X,i_+,i_-)$ to $(X',i'_+,i'_-)$. Then, $i_{\pm}(\del S)=i'_{\pm}(\del S)$ is an $n$-component link in $\del Z$ with framing induced from $i_{\pm}(S)$, or equivalently $i_{\pm}(S')$. Attaching $n$ two-handles along the components of this framed link we get a cobordism $W$ from $M$ to $M'$. Moreover, denote the union of co-cores of these $2$-handles by $F$. Then, $(W,F)$ is a homology cobordism from $(M,T)$ to $(M',T')$. Thus, the claim follows from Theorem \ref{thm:invhomcob}. 
\end{proof}
%
\begin{proof} (Proposition \ref{prop:hom})
First, by \ref{lem:coninv}, assigning $\Upsilon_{X}(\tsf)$ to every homology cobordism $(X,i_+,i_-)$ gives a map from $\mathcal{H}_{0,n}$ to $\R$.  Next,  Lemma \ref{lem:sum} implies that this map is a homomorphism. 
\end{proof}


\begin{example} For $n=1,2$, upsilon is equal to the known homomorphisms as follows.
 \begin{itemize}
\item $n=1$: $\mathcal{H}_{0,1}$  is the smooth homology cobordism group of integral homology spheres $\Theta^3_{\Z}$. As discussed in Example \ref{ex:closed} $\Upsilon$ is equal to the $d$-invariant.
\item $n=2$: $\mathcal{H}_{0,2}$  is the framed smooth concordance group of knots in integral homology spheres i.e. $\Z\oplus \mathcal{C}_{\Z}^{\smooth}$.  For any $t\in [0,2]$, $\Upsilon(t,2-t)$ is equal to the upsilon homomorphism of the corresponding oriented knot at $t$ after forgetting the framing.
\end{itemize}
\end{example}

\begin{lem}
The subgroup $\mathcal{M}_{0,n}$ of $\mathcal{H}_{0,n}$ is in the kernel of $\Upsilon$ for every $\tsf$. 
\end{lem}
\begin{proof}
It is obvious, because for any $\phi\in \mathcal{M}_{0,n}$ the underlying tangle of $X_\phi$ is the product tangle.
\end{proof}

\subsection{String link concordance group and upsilon} In this section, we show that $\Upsilon$ gives an obstruction for detecting whether two links are strongly concordant.

A \emph{string link} with $n$-components is a properly embedded $1$-manifold $L$ in $D^2\times [0,1]$ such that it has no closed components, and 
\[L\cap (D^2\times\{i\})=\{p_1,\cdots,p_n\}\times\{i\}\quad\quad\quad \text{for}\ i=0,1\]
and a fixed set $\{p_1, p_2, \cdots, p_n\}$ of $n$ distinct points in $D^2$. Moreover, \emph{closure} of $L$, denoted by $\widehat{L}$ is the $n$-component link in $S^3$ obtained by identifying $D^2\times \{0\}$ with $D^2\times\{1\}$ and filling the resulted torus boundary by attaching a solid torus so that every $p\times [0,1]$ with $p\in\del D^2$ bounds a disk in it. On the other hand, given an $n$ component link in $S^3$, along with an embedded disk that intersects each component of the link in exactly one point, one can construct a string link by removing a thickened neighborhood of the disk from $S^3$.

Given two string links $L$ and $L'$ in $D^2\times [0,1]$ a (smooth) \emph{concordance} from $L$ to $L'$ is a disjoint union of $n$ properly embedded disk $F=\coprod_{i=1}^nD_i$ in $D^2\times [0,1]\times [0,1]$ satisfying the followings:
\[
\begin{split}
&D_i\cap \left(D^2\times [0,1]\times\{0\}\right)=L_i\times\{0\}, \quad D_i\cap \left(D^2\times [0,1]\times\{1\}\right)=L'_i\times\{1\},\\ &\quad\quad\quad\quad\text{and}\quad D_{i}\cap\left(D^2\times\{0,1\}\times [0,1]\right)=\left(\coprod_{i=1}^n p_i\right)\times[0,1].
\end{split}
\]
Two string links are called \emph{concordant} is there exists a concordance between them. Further, two string links $L$ and $L'$ are (smoothly) concordant if and only if their closures are (smoothly) \emph{strongly concordant} \cite{habegger_lin_1998} i.e. there exists $n$ disjoint (smooth) properly embedded cylinders $\coprod_{i=1}^nS_i$ in $S^3\times [0,1]$ such that 
\[S_i\cap \left(S^3\times\{0\}\right)=\widehat{L}_i\quad\quad\text{and}\quad\quad S_i\cap\left(S^3\times\{1\}\right)=\widehat{L'}_i.\]

The set of string links with $n$ strands along with the concatenation operation forms a monoid. Moreover, being concordant is an equivalence relation, and taking the quotient of this monoid with this relation makes it into a group, called \emph{string link concordance group} and denoted by $\mathcal{C}(n)$. This group is first introduced in \cite{LeDimet-88}.

On the other hand, every string link with $n$ strands specifies a homology cylinder over $S_{0,n+1}$, and one can think of homology cylinders over $S_{0,n+1}$ as string links in homology cylinders over a disk. Further, if two string links in $D^2\times [0,1]$ are (smoothly) concordant, their corresponding homology cylinders are (smoothly) homology cobordant. So there is a homomorphism from $\mathcal{C}(n)$ to $\mathcal{H}^{\smooth}_{0,n}$.

Fix a distinguished boundary for $S=S_{0,n+1}$, for instance if $\del S=\coprod_{i=0}^n\del_iS$, suppose it is $\del_0S$. One can define the closure of a homology cylinder $(X,i_+,i_-)$ over $S$ as an $n-$component link in a homology sphere $Y$ as follows. As before, $Y$ is obtained from $X$ by first attaching two-handles along $\tau=i_+(\del S)=i_-(\del S)$ to construct $M$ and then filling the two sphere boundary components of $M$ with three-handles to get $Y$.  Consider pairwise disjoint properly embedded arcs $\gamma=\coprod_{i=1}^n\gamma_i$ on $S$ such that $\gamma_i$ connects $\del_iS$ to $\del_0S$ for $i=1,\cdots,n$. Then, $\coprod_{i=1}^n \left(i_{+}(\gamma_i)\cup i_-(\gamma_i)\right)$ is a link $L$ in $Y$, which is called the closure of $(X,i_+,i_-)$. It is not hard to show that $L$ does not depend on the choice of $\gamma$. 

Analogous to Habegger-Lin's result, one can show that two homology cylinders over $S_{0,n+1}$ are (smoothly) homologically cobordant if and only if their corresponding string links are (smoothly) strongly concordant in a homology cobordism between the corresponding homology spheres.

\begin{cor}
Given homology cylinders $X$ and $X'$ over $S_{0,n+1}$, if $\Upsilon_{X}(\tsf)\neq\Upsilon_{X'}(\tsf)$ then their link closures are not strongly concordant.
\end{cor}



 \subsection{Gluing operation and upsilon} In this section, we study how upsilon changes under the gluing operation defined in \cite[Section 2.3]{cha_friedl_kim_2011}. For some $n,n'>0$, let $X$ and $X'$ be homology cylinders over $S=S_{0,n}$ and $S'=S_{0,n'}$ respectively. Fix labelings $\del S=\coprod_{i=1}^n\del_iS$ and $\del S'=\coprod_{i=1}^{n'}\del_iS'$. Gluing $S$ and $S'$ by identifying $\del_nS$ and $\del_{n'}S'$ with an orientation reversing homeomorphism we get a surface of genus zero with $n+n'-2$ boundary components denoted by $\widetilde{S}$. Further, gluing $X$ and $X'$ along tubular neighborhoods of sutures corresponding to $\del_nS$ and $\del_{n'}S'$ we get a homology cylinder $\widetilde{X}$ over $\widetilde{S}$.

\begin{lem}\label{lem:gluing}
With the above notation fixed, for any $\widetilde{\tsf}\in[0,2]^{n+n'-2}$ with $t_1+t_2+\cdots+t_{n+n'-2}=2$ we have
\[\Upsilon_{\widetilde{X}}(\widetilde{\tsf})=\Upsilon_{X}(\tsf)+\Upsilon_{X'}(\tsf')\]
where $\tsf=(t_1,t_2,\cdots,t_{n-1},t_n+\cdots+t_{n+n'-2})$ and $\tsf'=(t_{n},t_{n+1},\cdots,t_{n+n'-2},t_1+t_2+\cdots+t_{n-1})$. 
\end{lem}
\begin{proof}
First, we prove a special case of this formula. Suppose $X'$ is the product cobordism over $S'$. Then, mapping $X$ to $\widetilde{X}$ is an injective group homomorphism $\mathcal{E}$ from $\mathcal{H}_{0,n}$ to $\mathcal{H}_{0,n+n'-2}$ \cite{cha_friedl_kim_2011}. A Heegaard diagram for the tangle corresponding to $\widetilde{X}$ is obtained from a Heegaard diagram for the corresponding tangle to $X$ by adding $n'-2$ base points $z_{n+1},\cdots,z_{n+n'-2}$ in the same domain as $z_n$. Using this diagram it is easy to see that 
\[\widetilde{t}\CFT(\widetilde{X})\simeq \tsf\CFT(X)\]
 and so $\Upsilon_{\widetilde{X}}(\widetilde{\tsf})=\Upsilon_{X}(\tsf)$.
 
 Suppose $X'$ is not a product cobordism over $S'$. Let $\widetilde{X}_1\in \mathcal{H}_{0,n+n'-2}$ be the homology cylinder over $\widetilde{S}$ obtained by the gluing of the product cobordism over $S$ and $X'$. Similarly, let $\widetilde{X}_2\in \mathcal{H}_{0,n+n'-2}$ be the result of gluing $X$ and the product cobordism over $S'$. It is easy to see that $\widetilde{X}=\widetilde{X}_1\cdot\widetilde{X}_2$ and so 
 \[\Upsilon_{\widetilde{X}}(\widetilde{\tsf})=\Upsilon_{\widetilde{X}_1}(\widetilde{\tsf})+\Upsilon_{\widetilde{X}_2}(\widetilde{\tsf})=\Upsilon_{X'}(\tsf')+\Upsilon_X(\tsf).\]
 \end{proof}

Following the construction of \cite{OSS-upsilon}, we construct a surjective homomorphism from $\mathcal{H}_{0,n}^\smooth$ to $\Z^\infty$ using the discontinuities of the derivatives of $\Upsilon_{X}(\tsf)$. More precisely, let $V^i=\sum_{j\neq i}V^i_j$. For any $\tsf\in\mathsf{L}_X$, let $l_{i}(\tsf)\subset \mathsf{L}_X$ denote intersection of the line that passes through $\tsf$ and is parallel to $V^i$ with $\mathsf{L}_X$. If $\tsf\notin\del l_i(\tsf)$, we define 
\[\Delta_i\Upsilon_{X}(\tsf)=D_{V^i}^+\Upsilon_X(\tsf)-D_{V^i}^-\Upsilon_X(\tsf)\]
where $D_{V^i}^{\pm}\Upsilon_X(\tsf)$ denote the derivative of $\Upsilon_X$ in the direction of $V_i$ from the right/left at $\tsf$. 

\begin{lem} \label{lem:int}
For any positive integer $1\le i\le n$ and real number $0\le a\le \frac{2}{n-1}$, let $\tsf^i_a\in\mathsf{L}_X$ be the solution for which $t_j=a$ if $j\neq i$ while $t_i=2-(n-1)a$. Then, $a\Delta_i\Upsilon_X(\tsf^i_a)$ is an even integer.
\end{lem}

\begin{proof}
Similar to the proof of \cite[Proposition 1.7]{OSS-upsilon}. Restrict $\Upsilon_X$ to the line $l_i(\tsf^i_a)$. There exists intersection points $\x$ and $\y$ such that $\gr_{\tsf_a^i}(\x)=\gr_{\tsf_a^i}(\y)$, and for $\tsf\in l_i(\tsf^i_a)$ near $\tsf^i_a$ we have $\Upsilon_{X}(\tsf)$ is equal to $\gr_{\tsf}(\x)$ in one side of $\tsf^i_a$ while equal to $\gr_{\tsf}(\y)$ on the other side. So,
\begin{align*}
a\Delta_i\Upsilon_X(\tsf^i_a)&=\sum_{j\neq i}a\left(\gr_j(\x)-\gr_j(\y)\right)-(n-1)a(\gr_i(\x)-\gr_i(\y))\\
&=2\left(\gr_i(\y)-\gr_i(\x)\right).
\end{align*}
\end{proof}


\begin{lem}\label{lem:surj}
For any $i=1,\cdots, n$, we define a homomorphism $f_i$ from $\HD_{0,n}^{\smooth}$ to $\Z^\infty$ by setting
\[X\mapsto \left\{\frac{1}{(2k+1)(n-1)}\Delta_i\Upsilon_X\left(\tsf^{i}_{\frac{2}{(2k+1)(n-1)}}\right)\right\}_{k=1}^\infty.\]
This homomorphism is surjective on the kernel of $\mathcal{H}_{0,n}^\smooth\to\mathcal{H}_{0,n}^{\topo}$.
\end{lem}
\begin{proof}
It follows from Lemmas \ref{lem:PL} and \ref{lem:int} that the target of this map is indeed $\Z^{\infty}$. To prove surjectivity, consider the map $\mathcal{E}_i$ from $\mathcal{H}_{0,2}\cong\mathcal{C}_{\Z}\oplus\Z$ to $\mathcal{H}_{0,n}$
defined by gluing any homology cylinder over $S_{0,2}$ i.e. a framed knot, to the product cylinder in $\mathcal{H}_{0,n}$ along a tubular neighborhood of the $i$th suture. Then, it follows from Lemma \ref{lem:gluing} that for any homology cylinder $X$ with $[X]=\mathcal{E}_i([K],m)$ for some knot $K\subset S^3$ and $m\in\Z$ we have
\[\Upsilon_X(\tsf)=\Upsilon_K(t_i)\quad\quad\quad\text{for}\quad{\tsf}=(t_1,t_2,\cdots,t_n)\in\mathsf{L}_X.\]
Therefore, 
\[\begin{split}
\frac{1}{(2k+1)(n-1)}\Delta_i\Upsilon_X\left(\tsf^{i}_{\frac{2}{(2k+1)(n-1)}}\right)&=\frac{1}{(2k+1)(n-1)}\left(-(n-1)\Delta_i\Upsilon_K\left(\frac{2}{2k+1}\right)\right)\\
&=-\frac{1}{2k+1}\Delta_i\Upsilon_K\left(\frac{2}{2k+1}\right)
\end{split}\]
and surjectivity follows from \cite[Theorems 1.17 and 1.20]{OSS-upsilon}.
\end{proof}

\newpage

\bibliographystyle{hamsalpha}
\bibliography{HFbibliography}

\end{document}